\documentclass[11pt]{amsart}

\usepackage{amsmath}
\usepackage{amsfonts}
\usepackage{amssymb}
\usepackage{amsthm}
\usepackage[alphabetic]{amsrefs}
\usepackage{bbold}
\usepackage{colonequals}
\usepackage{enumerate}
\usepackage[T1]{fontenc}
\usepackage{bbm}
\usepackage{mathtools}
\usepackage{chngcntr}
\usepackage{url,xspace,adjustbox}
\usepackage{footmisc}
\usepackage{amssymb}
\usepackage{geometry}

\newcommand{\GL}{\operatorname{GL}}

\newcommand{\Evs}{\operatorname{Evs}}

\newcommand{\RPsi}{\operatorname{R\hskip-1pt\Psi}}

\newcommand{\Der}{\operatorname{D}}
\newcommand{\Per}{\operatorname{Per}}
\newcommand{\Loc}{\operatorname{Loc}}
\newcommand{\Rep}{\operatorname{Rep}}
\newcommand{\Lift}{\operatorname{Lift}}
\newcommand{\LS}{\operatorname{LS}}
\newcommand{\trace}{\operatorname{trace}}
\newcommand{\rank}{\operatorname{rank}}
\newcommand{\nocontentsline}[3]{}
\newcommand{\tocless}[2]{\bgroup\let\addcontentsline=\nocontentsline#1{#2}\egroup}
\newcommand{\Lgroup}[1]{\prescript{L}{}#1}
\newcommand{\CC}{\mathbb{C}}
\newcommand{\St}{\text{St}}

\newcommand{\ABV}{{\mbox{\raisebox{1pt}{\scalebox{0.5}{$\mathrm{ABV}$}}}}}
\newcommand{\op}[1]{\operatorname{#1}}
\newcommand{\1}{{\mathbbm{1}}}
\newcommand{\dualgroup}[1]{\widehat{#1}}
\newcommand{\IC}{{\mathcal{IC}\hskip-1pt }}
\newcommand{\Sym}{\operatorname{Sym}}
\newcommand{\Ind}{\operatorname{Ind}}

\newcommand{\diag}{\operatorname{diag}}
\newcommand{\tq}{\ \vert \ }
\newcommand{\SL}{\operatorname{SL}}
\newcommand{\Lie}{\operatorname{Lie}}
\newcommand{\Ad}{\operatorname{Ad}}
\newcommand{\KS}{{\operatorname{KS}}}

\usepackage{adjustbox}
    
\newtheorem{theorem}{Theorem}[section]
\newtheorem{theorem*}{Theorem}
\newtheorem{proposition}[theorem]{Proposition}
\newtheorem{hypothesis}[theorem]{Hypothesis}

\theoremstyle{definition}

\newtheorem{lemma}[theorem]{Lemma}

\newtheorem{example}[theorem]{Example}
\newtheorem{remark}[theorem]{Remark}

\bibstyle{alphanumeric}

\usepackage{soul}

\usepackage{mathrsfs,stmaryrd}
\usepackage{yfonts, bbm}
\usepackage{enumitem}
\usepackage{float}

\usepackage{hyperref}
\hypersetup{
  colorlinks   = true, 
  urlcolor     = blue, 
  linkcolor    = red, 
  citecolor   = cyan 
}

\usepackage{tikz}
\usetikzlibrary{shapes,arrows,calc,matrix}
\usepackage{tikz-cd}
\usepackage{mathdots}

\usepackage{todonotes}
\usepackage{color}

\usepackage{amsmath}
\usepackage{lipsum}
\usepackage{setspace}

\counterwithin{table}{section}

\tikzset{
  symbol/.style={
    draw=none,
    every to/.append style={
      edge node={node [sloped, allow upside down, auto=false]{$#1$}}}
  }
}

\title[Vogan's conjecture on Arthur packets for $\mathop{GL}_n$ over $p$-adic fields]{Proof of Vogan's conjecture on Arthur packets\\ for $\mathop{GL}_n$ over $p$-adic fields}

\author[C. Cunningham]{Clifton Cunningham}
\address{Department of Mathematics and Statistics, University of Calgary, 
2500 University Drive NW, 
Calgary, Alberta, 
T2N 1N4, 
Canada}
\email{clifton@automorphic.ca}
\thanks{Cunningham's research is supported by NSERC Discovery Grant RGPIN-2020-05220. He is also grateful to the \href{www.fields.utoronto.ca}{Fields Institute for Research in Mathematical Sciences} where some of this work was conducted and to Casa Matemática Oaxaca where it was first presented at a BIRS-CMO workshop.}

\author[M. Ray]{Mishty Ray}
\address{Department of Mathematics and Statistics, University of Calgary, 
2500 University Drive NW, 
Calgary, Alberta, 
T2N 1N4, 
Canada}
\email{mishty.ray@ucalgary.ca}
\thanks{Ray thanks café Stable in downtown Calgary for its excellent research environment of snowy concrete slabs and unintentionally slowed down Taylor Swift vinyl recordings.}

\date{\today}

\begin{document}

\begin{abstract}
In this paper we prove Vogan's conjecture on Arthur packets for general linear groups over $p$-adic fields, building on earlier work. The proof uses a special case of endoscopic lifting, adapted from the 1992 book by Adams, Barbasch and Vogan, where it was articulated for real groups. 
%
%
\end{abstract}

\maketitle

\tableofcontents

\section{Introduction}\label{introduction}

Thirty years ago, David Vogan conjectured a purely local description of A-packets for $p$-adic groups \cite{Vogan:Langlands}, closely related to a more developed theory for real groups by Adams, Barbasch and Vogan \cite{ABV}. 
While there is considerable evidence in the form of examples from \cite{CFMMX}*{Chapters 11-16}, Vogan's conjecture for $p$-adic groups remains open.
In this paper we prove this conjecture for general linear groups over $p$-adic fields (non-archimedean fields of charactersitic $0$), building on previous work in \cite{CR:irred}. We view this as a step toward proving Vogan's conjecture for the groups treated by Arthur in \cite{Arthur:book}, and the strategy of the proof given here reflects this objective. 

This conjecture was explicated in \cite{CFMMX}*{Conjecture~1} for a quasisplit symplectic or orthogonal $p$-adic group $G$ and attributed to Vogan. It predicts that, for every Arthur parameter \[\psi : W_F''\coloneqq W_F \times \SL_2(\CC) \times \SL_2(\CC) \to \Lgroup{G},\] the A-packet $\Pi_\psi(G)$ of representations of $G(F)$ coincides with the ABV-packet attached to the Langlands parameter $\phi_\psi$ determined by $\psi$. As defined \cite{CFMMX}*{Definition~1} and recalled in \cite{CR:irred}, the ABV-packet for $\phi_\psi$ is given by 
\[
\Pi^\ABV_{\phi_\psi}(G) \coloneqq
\left\{
\pi\in \Pi_{\lambda}(G) \tq 
\Evs_\psi \mathcal{P}(\pi) \ne 0
\right\},
\]
where 
\begin{itemize}
    \item
    $\Pi_{\lambda}(G)$ is the set of equivalence classes of irreducible, smooth representations of $G(F)$ with  infinitesimal parameter $\lambda$ determined by $\psi$;
    \item
    $\mathcal{P}(\pi)$ is the simple object in the category $\Per_{H_\lambda}(V_\lambda)$ of equivariant perverse sheaves on the moduli space $V_\lambda$ of Langlands parameters of matching $\pi$ under the enhanced local Langlands correspondence, with $H_\lambda \coloneqq Z_{\dualgroup{G}}(\lambda)$; 
    \item
    $\Evs_\psi$ is the functor 
    $$
    \Evs_\psi : \Per_{H_\lambda}(V_\lambda) \to \Loc_{H_\lambda}(T^*_{C_\psi}(V_{\lambda})^\text{reg}) \equiv \Rep(A_\psi),
    $$
    introduced in \cite{CFMMX}*{Section 7.10}, where $C_\psi$ is the $H_\lambda$-orbit of the point in $V_\lambda$ corresponding to $\phi_\psi$ and where $T^*_{C_\psi}(V_{\lambda})^\text{reg}$ is the regular part of the conormal bundle $T^*_{C_\psi}(V_\lambda)$.
\end{itemize}
These terms are all defined carefully in \cite{CFMMX} and revisited in \cite{CR:irred}.

The main result of this paper, Theorem~\ref{thm:main}, shows that
\begin{equation}\label{equation:main}
\Pi^\ABV_{\phi_\psi}(G) = \Pi_\psi(G),
\end{equation}
for every Arthur parameter $\psi$ for $G$, for $G = \GL_n$. 

Let us now sketch the proof Theorem~\ref{thm:main}. We begin with an arbitrary Arthur parameter $\psi$ of $G$, a map $W''_F \to \GL_n(\CC)$. This map, in general, has the form 
\[
\psi = \psi_1\oplus\cdots\oplus\psi_k.
\] 
Here each $\psi_i$ is an \textit{irreducible} Arthur parameter (irreducible as a representation of $W_F''$; see \cite{CR:irred}*{Section 6} for detailed definition). Set $m_i \coloneqq \dim \psi_{i}$. 
This decomposition naturally picks out a Levi subgroup $M \simeq \GL_{m_1}\times \cdots \times \GL_{m_k}$. Observe that $\dualgroup{M}$ is a Levi subgroup of $\dualgroup{G}$ containing the image of $\psi$.
Pick $s\in \dualgroup{G}$, of finite order, and therefore semisimple, so that $Z_{\dualgroup{G}}(s) = \dualgroup{M}$.
Let $\psi_M : W''_F \to \dualgroup{M}$ be the Arthur parameter for $M$ such that the following diagram commutes.
\[
\begin{tikzcd}
W''_F \arrow{rr}{\psi} \arrow{dr}[swap]{\psi_M}  && \dualgroup{G}\\
& \dualgroup{M} =Z_{\dualgroup{G}}(s) \arrow{ur} 
\end{tikzcd}
\]
Observe that $\psi_M$ is an irreducible Arthur parameter for $M$. 
Let $\lambda_M$ be the infinitesimal parameter of $\psi_M$.
Now the inclusion $Z_{\dualgroup{G}}(s) \to \dualgroup{G}$ induces an inclusion
\[
\varepsilon : V_{\lambda_M} \hookrightarrow V_\lambda,
\]
which is equivariant for the action of $H_{\lambda_M}\coloneqq Z_{\dualgroup{M}}(\lambda_M)$ on $V_{\lambda_M}$ and the action of $H_\lambda \coloneqq Z_{\dualgroup{G}}(\lambda)$ on $V_\lambda$.
Indeed, 
\[
V_{\lambda_M} = V_\lambda^s = \{ x\in V_{\lambda} \tq \Ad(s)x = x\}.
\]
Because $\psi_M$ is irreducible, earlier work \cite{CR:irred} establishes Vogan's conjecture for $\psi_M$:
\[
\Pi^\ABV_{\phi_{\psi_M}}(M) = \Pi_{\psi_M}(M).
\]

In order to lift this result from $M$ to $G$ we use the local Langlands correspondence to define a pairing between two Grothendieck groups that appear naturally on either side of the correspondence: on the spectral side, the category $\Rep^\text{fl}_{\lambda}(G)$ of finite-length representations of $G(F)$ with infinitesimal parameter $\lambda$ determined by $\psi$; on the Galois/geometric the category $\Per_{H_{\lambda}}(V_{\lambda})$ of $H_{\lambda}$-equivariant perverse sheaves on $V_{\lambda}$.
In Section~\ref{subsection:eta} we recall a virtual representation $\eta^{\Evs}_{\psi} \in K\Rep^\text{fl}_\lambda(G)$ that characterizes the ABV-packet $\Pi^\ABV_{\phi_\psi}(G)$, with the property that 
\[
\langle \eta^{\Evs}_{\psi}, [\mathcal{F}] \rangle_{\lambda} 
=
(-1)^{d(\psi)} \rank\left(\Evs_\psi \mathcal{F}\right), \qquad \forall \mathcal{F}\in \Per_{H_\lambda}(V_\lambda),
\]
where $d(\psi)$ is the dimension of the $H_\lambda$-orbit of $x_\psi$ in $V_\lambda$; see Proposition~\ref{prop:eta}.
Since A-packets are singletons for $\GL_n$, we set $\eta_{\psi} :=[\pi_{\psi}]$. Thus, to show that the ABV-packet coincides with this A-packet, it is enough to show
\[\eta^{\Evs}_{\psi}=\eta_{\psi}=[\pi_{\psi}].\]
As the pairing $\langle \cdot, \cdot \rangle_{\lambda}$ is non-degenerate, this is equivalent to showing that 
\[
\langle \eta^{\Evs}_{\psi}, [\mathcal{F}] \rangle_{\lambda} = \langle \eta_{\psi}, [\mathcal{F}] \rangle_{\lambda},
\]
for all $\mathcal{F}\in \Per_{H_\lambda}(V_\lambda)$. We do this by showing the three equalities below, for all $ \mathcal{F}\in \Per_{H_\lambda}(V_\lambda)$.
\[
\begin{tikzcd}
    {\langle \eta^{\Evs}_{\psi}, [\mathcal{F}] \rangle_{\lambda}}
    \arrow[equal]{d}[swap]{\text{Fixed-point formula \hskip3pt}}{\text{Prop.~}\ref{prop:M}}
    \arrow[dashed,equal]{rrrr}{\text{Vogan's conjecture for $\psi$}}[swap]{\text{Theorem~\ref{thm:main}}}
    &&&& 
    {\langle \eta_{\psi}, [\mathcal{F}] \rangle_{\lambda}}
    \\
    \langle \eta^{\Evs}_{\psi_M}, [\mathcal{F}\vert_{V_{\lambda_M}}] \rangle_{\lambda_M}
    \arrow[equal]{rrrr}[swap]{\text{Vogan's conjecture for $\psi_M$}}{\text{Prop.~}\ref{prop:M}}
    &&&& 
    \langle \eta_{\psi_M}, [\mathcal{F}\vert_{V_{\lambda_M}}] \rangle_{\lambda_M} 
    \arrow[equal]{u}[swap]{\text{\hskip3pt Endoscopic lifting}}{\text{Prop.~}\ref{prop: Lift on irreducibles of A type}}
\end{tikzcd}
\]

Even for general linear groups, non-Arthur type ABV-packets hold some surprises. 
Specifically, \cite{CFK} presents a non-Arthur type Langlands parameter $\phi_\KS$ for $\GL_{16}$ such that $\Pi^\ABV_{\phi_\KS}(\GL_{16})$ consists of two representations. 
Even more remarkably, the coronal representation $\pi_\psi$ in $\Pi^\ABV_{\phi_\KS}(\GL_{16})$ is of Arthur type. 
The main result of this paper implies $\Pi^\ABV_{\phi_\psi}(\GL_{16}) = \Pi_\psi(\GL_{16})$. 
This is an example, then, of the following containments.
\[
\begin{tikzcd}
& \Pi^\ABV_{\phi_\KS}(\GL_{16}) = \{ \pi_{\phi_\KS}, \pi_\psi \} & \\
\arrow[>->]{ur} \Pi_{\phi_\KS}(\GL_{16}) = \{ \pi_{\phi_\KS} \} && \arrow[>->]{ul} \{ \pi_\psi\} = \Pi_\psi(\GL_{16}) 
\end{tikzcd}
\]

We remark that ABV-packets for $p$-adic groups were not introduced in \cite{ABV}, since that book treats real groups only. In \cite{ABV}*{Theorem 24.8}, the authors define the functor $Q^\text{mic}$ using stratified Morse theory, and this can be used to characterize the microlocal packet attached to an L-parameter, as defined in Definition 19.5 of \textit{loc. cit.}. 
A similar approach is taken in \cite{Vogan:Langlands} which does treat $p$-adic groups, but uses a functor patterned after $Q^\text{mic}$.
We use the functor $\Evs$ instead, as explained above. 
We use the name ABV-packets simply to acknowledge the debt owed to the authors for the development of this theory. 
Likewise, we use the term \textit{``Vogan's conjecture''} for \cite{CFMMX}*{Conjecture 1} as it arose out of Vogan's work in \cite{Vogan:Langlands}. 

This paper deals with non-archimedean local fields $F$ of characteristic $0$ as the notion of A-packets for non-archimedean local fields of nonzero characteristic is unclear, to the best of our knowledge; however the Local Langlands correspondence and the geometric perspective exists for such fields and we can define an ABV-packet for characteristic nonzero fields in exactly the same manner. Thus, we propose that ABV-packets generalize local A-packets in the sense that the ABV-packet for a Langlands parameter of A-type can be used as an analogue to the corresponding A-packet for a non-archimedean local field of any characteristic. 


\subsection{Acknowledgements}
Some ideas in this paper are inspired by \cite{ABV} which treats real groups; we are happy to acknowledge these three authors, especially David Vogan, for his continued support for this project.
We also thank the entire Voganish Project research group, especially Bin Xu, Geoff Vooys, and Kristaps Balodis, for their contributions to this research. We also thank Matthew Sunohara, Tom Haines, and Peter Dillery for helpful conversations.

\subsection{Relation to other work}

This paper is part of the Voganish Project; for related results, we refer to \cite{CFMMX}, \cite{Mracek}, \cite{CFK}, \cite{CFZ:cubics}, \cite{CFZ:unipotent}, and \cite{CR:irred}. 



The main result of this paper, Theorem~\ref{thm:main}, can be proved by an argument different than the one presented here, as we now explain.
In \cite{CR:irred}*{Lemma 5.1} we proved the following geometric statement: if $\psi$ is a simple (irreducible with trivial restriction to $W_F$) Arthur parameter and $C_\psi$ is its associated $H_\lambda$-orbit in $V_\lambda$, then, for any $H_\lambda$-orbit $C$ in $V_\lambda$, $C_\psi \leq C$ and $C^*_\psi \leq C^*$ implies $C_\psi = C$; here we refer to the Zariski-closure relation on these orbits. In his MSc thesis written under the supervision of Andrew Fiori, Connor Riddlesden \cite{Riddlesden} extended this result to unramified Arthur parameters $\psi$, dropping the irreducibility condition appearing in \cite{CR:irred}*{Lemma 5.1}, though not applying to arbitrary Arthur parameters. When combined with unramification as it appears in \cite{CR:irred}*{Section 6, especially Lemma 6.8} and results from \cite{CFMMX}, these can be assembled to give an alternate proof of Theorem~\ref{thm:main}. 
While our proof of Theorem~\ref{thm:main} is perhaps more complicated than this alternate argument, we believe that our strategy is better adapted to generalizations, specifically, to proving Vogan's conjecture on A-packets to quasisplit classical groups and their pure inner forms. This belief is based on the fact that, in this paper, we have used endoscopic lifting
in a very special case. While it is characterized by parabolic induction in this case, we expect Langlands-Shelstad transfer to play a role more generally. 
Since Arthur's packets are characterized by Langlands-Shelstad transfer and Kottwitz-Shelstad transfer, together with certain normalizing choices referring to Whittaker models, we expect the geometric incarnation of both kinds of transfer to play an important role in extending our main result to other groups $G$.

\subsection{Notation}\label{subsection:notation}
In this work, for the most part, we follow notational conventions established in \cite{CR:irred}.
Here, $F$ is a non-archimedean local field of charactersitic $0$, also known as a $p$-adic field.
Henceforth, $G$ is $\GL_n$ and $P$ is a parabolic subgroup of $G$ with Levi subgroup $M$ and unipotent radical $N$; these statements are made in the category of algebraic groups over $F$, not their $F$-points, for which we use the notation $G(F)$, $P(F)$, etc. 

We use the notation 
\[
W_F' \coloneqq W_F\times \SL_2(\CC);
\]
this topological group is denoted by $L_F$ in Arthur's work.
We also use the notation 
\[
W_F'' \coloneqq W_F\times \SL_2(\CC)\times \SL_2(\CC);
\]
this topological group is denoted by $L'_F$ in Arthur's work.

Let $\dualgroup{G}$ denote the complex dual group of $G$, which for us is simply $\GL_n(\CC)$. 

By a Langlands parameter we mean an admissible homomorphism $\phi: W'_F \to \dualgroup{G}$, as defined in \cite{Borel:Corvallis}, for example.
We refer to $\dualgroup{G}$-conjugacy classes of Langlands parameters as L-parameters.

An infinitesimal parameter $\lambda: W_F\to \dualgroup{G}$ is simply a Langlands parameter with domain $W_F$.
The infinitesimal parameter $\lambda_\phi$ of a Langlands parameter $\phi$ is defined by 
\[
\lambda_\phi(w) \coloneqq \phi(w,\diag(|w|^{1/2}, |w|^{-1/2})).
\]

By an Arthur parameter we mean a homomorphism $\psi: W''_F \to \dualgroup{G}$ satisfying conditions explained in \cite{CFMMX}*{Section 3.5}, 
notably, that its restriction to $W_F$ is bounded.
The Langlands parameter $\phi_\psi$ is defined by $$\phi_\psi(w,x)\coloneqq \psi(w,x,\diag(|w|^{1/2}, |w|^{-1/2})).$$ 
The infinitesimal parameter $\lambda_\psi$ of $\psi$ is the infinitesimal parameter of $\phi_\psi$, thus given by \[
\lambda_\psi(w) \coloneqq \psi(w,\diag(|w|^{1/2}, |w|^{-1/2}),\diag(|w|^{1/2}, |w|^{-1/2})).
\]
When $\psi$ has been fixed, we set $\lambda=\lambda_\psi$.

For a smooth irreducible representation $\sigma$ of $M(F)$, the symbol $\Ind_P^G(\sigma)$ denotes the normalized parabolic induction of the representation $\sigma$ of the $F$-rational points $M(F)$ of the Levi subgroup $M$ of $P$; this means we inflate $\sigma$ from $M(F)$ to $P(F)$, twist by the modulus quasicharacter $\delta_P^{1/2}$ for the parabolic $P(F)$ and then induce from $P(F)$ to $G(F)$.

As in our earlier work, we use the notation $D_{H_\lambda}(V_\lambda)$ for the $H_\lambda$-equivariant derived category of $\ell$-adic sheaves on $V$; see  \cite{Lusztig:Cuspidal2}*{Section 1.10}, \cite{Achar:book}, or especially \cite{Vooys}*{Definition 3.1.14} for the definition of this category and \cite{Vooys}*{Theorem 9.0.38} for how to navigate different perspectives on this category.

In this paper we write $\rank(E)$ for the Euler characteristic of a graded vector space $E = \oplus_{i\in \mathbb{Z}} E^i$:
\[
\rank(E) = \sum_{i\in \mathbb{Z}}(-1)^i \dim(E^i).
\]
For $\mathcal{F}\in \Der_{H_\lambda}(V_\lambda)$ and $x\in V_\lambda$, we write $\mathcal{F}_x \in \Der_{Z_H(x)}(x)$ for the stalk of $\mathcal{F}$ at $x$ and $\mathcal{H}^\bullet_x\mathcal{F}$ for its cohomology complex, often viewed as a graded vector space.



We follow the conventions of \cite{BBD} regarding perverse sheaves; in particular, the restriction of $\IC(\mathcal{L}_C)$ to $C$ is $\mathcal{L}_C[\dim C]$, and complexes are shifted according to $(\mathcal{F}[n])^i = \mathcal{F}^{n+i}$.
The notation $\1_X$ is used to denote the constant sheaf on $X$.

\section{Preliminaries on Vogan's conjecture on A-packets}\label{section:eta}

In this section we revisit the definition of ABV-packets for $p$-adic groups from \cite{CFMMX}, also recalled in \cite{CR:irred} and cast it in a form that is adapted to the proof of the main result, Theorem~\ref{thm:main}. 
Instead of purely working over $\Pi_{\lambda}(G)$ and $\Per_{H_{\lambda}}(V_{\lambda})_{/\op{iso}}^{\op{simple}}$ as in the previous paper \cite{CR: irred}, we work over Grothendieck groups $K\Rep_{\lambda}^{\op{fl}}(G)$ and $K\Per_{H_{\lambda}}(V_{\lambda})$. We therefore spend some time explaining these groups. 


\subsection{Spectral side}

By the Langlands correspondence, the $\dualgroup{G}$-conjugacy class of $\lambda$ is identified with a cuspidal support $(L,\sigma)_{G} \in \Omega(G)$, the Bernstein variety for $G(F)$. We remark that $L$ is a Levi subgroup of $M$, where $M$ is determined by $\psi$ as in Section~\ref{introduction}.
Let $\Rep_\lambda(G)$ be the \emph{cuspidal support category} of smooth representations of $G(F)$ whose Jordan-Holder series is contained in the Jordan-Holder series of $\mathop{\text{Ind}}_P^G(\sigma)$ where $M$ is a parabolic subgroup of $G$ with  Levi component $M$.

Let $\Rep^\text{fl}_\lambda(G)$ be the subcategory of finite-length representations. The Grothendieck group $K\Rep_{\lambda}^{\op{fl}}(G)$ has two different bases - one consisting of smooth irreducible representations of $G(F)$ that share the infinitesimal parameter $\lambda$, and the other of standard representations attached to these irreducible representations. We use both these bases in the rest of the paper, so we recall the theory surrounding these objects below. 

\subsubsection{Irreducible representations}
Let $\Pi_\lambda(G) = \{ \pi_i \tq i\in I \}$ be the Jordan-Holder series of $\mathop{\text{Ind}}_P^G(\sigma)$;
this may be identified with the set of isomorphism classes of irreducible admissible representations of $G(F)$ with infinitesimal parameter $\lambda$. 

Smooth irreducible representations of $G(F)$ are classified using Zelevinsky theory \cite{Z2}. 
This was surveyed beautifully in \cite{kudla1994local} and we use his notation. For any representation $\pi$ of $\GL_m(F)$, let $\pi(i):=|\text{det}(\cdot)|^i\pi$. 
For a partition $n=\underbrace{m+m+\ldots+m}_{r\text{-times}}$ and a supercuspidal representation $\sigma$ of $\GL_m(F)$, we call 
\begin{equation}
    \label{segment}
    (\sigma, \sigma(1), \ldots, \sigma(r-1))=[\sigma,\sigma(r-1)]=:\Delta
\end{equation} 
a segment. 
This segment determines a representation of a standard parabolic subgroup $P$ of $G$ whose Levi subgroup is identified with $\underbrace{\GL_m \times \GL_m\times \cdots \times \GL_m}_{r\text{-times}}$. 
We can then carry out parabolic induction to obtain the induced representation $\Ind_P^{G}(\sigma \otimes \sigma(1)\otimes \cdots \otimes \sigma(r-1))$ of $G$, which has a unique irreducible quotient denoted by $Q(\Delta)$. We refer to $Q(\Delta)$ as the \textit{Langlands quotient} associated to $\Delta$. For a segment $\Delta=[\sigma,\sigma(r-1)]$, we set 
\[\Delta(x)\coloneqq [\sigma(x),\sigma(r-1+x)].\]
A multisegment is a multiset of segments. A segment $\Delta_1$ is said to \textit{precede} $\Delta_2$ if $\Delta_1 \not \subset \Delta_2$, $\Delta_2 \not \subset \Delta_1$, and there exists a positive integer $x$ so that \[\Delta_2=\Delta_1(x) \]
making $\Delta_1 \cup \Delta_2$ a segment. A multisegment $\{\Delta_1,\Delta_2,\cdots, \Delta_k\}$ where $\Delta_i$ does not precede $\Delta_j$ for any $i<j$ is said to satisfy the "does not precede" condition. 

Now let $\alpha=\{\Delta_1, \Delta_2, \ldots, \Delta_k\}$ be a multisegment satifying the "does not precede" condition. Let $P'$ denote the standard parabolic subgroup specified by $\alpha$. The Langlands classification theorem tells us that any smooth irreducible representation of $G$ occurs as a unique irreducible quotient of the parabolically induced representation $\Ind_{P'}^G(Q(\Delta_1)\otimes \cdots \otimes Q(\Delta_k))$ -- we denote that quotient by $Q(\Delta_1, \Delta_2, \ldots, \Delta_k)$ or $Q(\alpha)$ and refer to it as the \textit{Langlands quotient} associated to $\alpha$; see \cite{kudla1994local}*{Theorem 1.2.5}. Next, for integers $i<j$ we introduce the notation 
\begin{equation}
    \label{oursegment}
    [i,j]:=(|\cdot|^i, |\cdot|^{i+1}, \ldots, |\cdot|^j)
\end{equation}
 for a segment which is the special case of \eqref{segment} when we consider the partition $1+1+ \cdots +1$ and $\sigma$ to be the character $|\cdot|$ of $F^{\times}$. This notation may be extended to half integers $i<j$ as long as $j-i+1$ is a positive integer (this is the length of the segment). A segment of length 1 of the form $\{|\cdot|^i\}$ is just denoted $[i]$. 

\subsubsection{Standard representations}\label{ssec: standard representations}
Next, we review the notion of a standard representation, also known as standard module in the literature. While this is written down in many different places, we follow the exposition of \cite{Konno}. A standard representation of $G(F)$ corresponds to the data $(P,\nu, \tau)$, where $P=MN$ is a standard parabolic subgroup of $G$, $\nu \in \mathfrak{a}_{P}^{*,+}$, and $\tau$ a tempered representation of $M$. The definition of $\mathfrak{a}_{P}^{*,+}$ is given in Section 2.2 of \textit{loc. cit.}. The character $\nu$ corresponds to a $P$-positive unramified quasicharacter $\exp{\nu}$ of $M(F)$ as explained in Section 2.3 of \textit{loc. cit.}. The standard representation associated to this data is given by $\Ind_P^{G}(\tau \otimes  \exp{\nu} )$. This representation has a unique irreducible quotient (see Corollary 3.2 of \textit{loc. cit.}), say $\pi$. In this paper we use the notation
\[
\Delta(\pi) := \Ind_P^{G}(\tau \otimes  \exp{\nu} ),
\]
and call it the standard representation of $\pi$. 

Thus, $K\Rep^{\op{fl}}_{\lambda}(G)$ has two $\mathbb{Z}$-bases - one given by irreducible representations 
\[\{[\pi]: \pi \in \Pi_{\lambda}(G)\},\]
and the other given by standard representations
\[\{[\Delta(\pi)]: \pi \in \Pi_{\lambda}(G)\}.\] We note that the latter is true because every irreducible representation is the unique quotient of its standard representation, by the Langlands classification theorem.

\subsection{Galois/geometric side}

Recall that in this paper we make free use of \cite{CFMMX} and \cite{CR:irred}. 
In particular, for every infinitesimal parameter $\lambda : W_F \to \dualgroup{G}$, set
\[
V_\lambda \coloneqq \{ x\in \Lie \dualgroup{G} \tq \Ad(\lambda(w))(x) = |w| x,\  \forall w\in W_F \},
\]
and 
\[
H_\lambda \coloneqq \{ g\in \dualgroup{G} \tq \lambda(w) g \lambda(w)^{-1} = g,\  \forall w\in W_F \}.
\]
Then $V_\lambda$ is a prehomogeneous vector space for the $H_\lambda$-action inherited from conjugation in $\Lie\dualgroup{G}$, stratified into $H_\lambda$-orbits $\{ C_i \tq i \in I\}$.
Recall that $V_\lambda$ is a moduli space of Langlands parameters with infinitesimal parameter $\lambda$. Vogan's geometric perspective relates Langlands parameters to simple objects up to isomorphism in the category $\Per_{H_{\lambda}}(V_{\lambda})$. As on the spectral side, $K\Per_{H_{\lambda}}(V_{\lambda})$ has two bases - one consisting of simple perverse sheaves and the other of standard sheaves, which we explain below.

\subsubsection{Simple perverse sheaves}\label{simple perverse sheaves}
Simple objects in $\Per_{H_\lambda}(V_\lambda)$ are all of the form $\IC(\mathcal{L}_C)$, where $C$ is an $H_{\lambda}$-orbit in $V_{\lambda}$, and $\mathcal{L}_C \in \Loc_{H_\lambda}(V_\lambda)$ is a simple equivariant local system. For $G=\GL_n$, simple perverse sheaves are of the form $\IC(\1_C)$ as each orbit only has the trivial irreducible local system. We invoke the local Langlands correspondence for $G$ and write $\pi_\phi \in \Pi(G)$ for the isomorphism class of irreducible representation with Langlands parameter $\phi$, and likewise $C_{\pi}$ for the $H_\lambda$-orbit in $V_\lambda$ of parameters that correspond to $\pi$. We refer to the latter identification as the Vogan-Langlands correspondence in \cite{CR:irred}. We summarize the entire correspondence below:
\[    
\begin{tikzcd}
	{\Pi_{\lambda}(G(F))} & {\Phi_{\lambda}(G(F))} & {H_\lambda \operatorname{-orbits} \operatorname{ in}V} & {\operatorname{Perv}_{H_\lambda}(V_\lambda)^{\operatorname{simple}}_{/\operatorname{iso}},} & {} \\[-10pt]
	{\pi \hspace{2mm}} & \phi & {\hspace{2mm}C_{\phi} \hspace{2mm}} & {\hspace{2mm}\mathcal{P}(\pi)=\mathcal{IC}(\mathbb{1}_{C_{\phi}}).}
	\arrow[from=1-1, to=1-2]
	\arrow[from=1-2, to=1-3]
	\arrow[maps to, from=2-1, to=2-2]
	\arrow[maps to, from=2-2, to=2-3]
	\arrow[from=1-3, to=1-4]
	\arrow[maps to, from=2-3, to=2-4]
\end{tikzcd}
\]
Thus, there is a unique orbit $C_{\phi_{\psi}}$ in $V_\lambda$ attached to $\phi_{\psi}$. We shorten this notation to $C_{\psi}$. 
We write $\mathcal{P}(\pi)$ for $H_\lambda$-equivariant intersection cohomology complex on $V_\lambda$ determined by $\pi$ through $C_\pi$; note that $\mathcal{P}(\pi)$ is a simple simple object in the abelian category $\Per_{H_\lambda}(V_\lambda)$ of $H_\lambda$-equivariant perverse sheaves on $V_\lambda$.

\subsubsection{Standard sheaves}\label{standard sheaves}

For any $H_\lambda$-orbit $C\subseteq V_\lambda$ and any simple local system $\mathcal{L}_C$ on $C$, we introduce the notation $\mathcal{L}_C^\natural$ for the $H_\lambda$-equivariant sheaf on $V_\lambda$ with the defining property 
\[
\left(\mathcal{L}_C^\natural\right)\vert_{C'} 
=
\begin{cases}
    \mathcal{L}_C & C'=C,\\
    0 & C'\ne C.
\end{cases}
\] 
The Grothendieck group $K\Per_{H_{\lambda}}(V_{\lambda})$ has two $\mathbb{Z}$-bases, corresponding to the two t-structures in play: one consisting of simple perverse sheaves 
\[
\{[\mathcal{IC}(\1_{C})]: \text{ }C \text{ ranges over }H_{\lambda}\text{-orbits in }V_{\lambda}\}
\]
and the other of standard sheaves 
\[
\{[\1_C^\natural]: \text{ }C \text{ ranges over }H_{\lambda}\text{-orbits in }V_{\lambda}\};
\]
see \cite{Achar:book}*{Proposition A.9.5} and also \cite{ABV}*{p.12, paragraph 1}, for a related instance of this phenomenon.

\subsection{Dual Grothendieck groups}\label{subsection:Pairing}

For every infinitesimal parameter $\lambda :W_F \to \dualgroup{G}$, the local Langlands correspondence determines a perfect pairing between Grothendieck groups
\[
    K\Rep^\text{fl}_\lambda(G) \times K\Per_{H_\lambda}(V_\lambda) \to \mathbb{Z}
\]
defined by
\begin{equation}\label{eqn:pairing}
\langle [\pi],[\mathcal{P}]\rangle_{\lambda} =
\begin{cases}
    (-1)^{d(\pi)} & [\mathcal{P}] = [\mathcal{P}(\pi)],\\
    0 & \text{otherwise},
\end{cases}
\end{equation}
where $d(\pi) \coloneqq \dim C_{\pi}$.
Recall that $\mathcal{P}(\pi)$ denotes a simple object in $\Per_{H_\lambda}(V_\lambda)$ matching $\pi\in \Pi_\lambda(G)$ as in Section \ref{simple perverse sheaves}.
Notice that this perfect pairing between Grothendieck groups matches $\pi\in \Pi_\lambda(G)$ with the shifted perverse sheaf $\mathcal{P}(\pi)[-\dim C_\pi]$.
 
If we index $\Pi_\lambda(G) = \{ \pi_i \tq i\in I\}$ and likewise index isomorphism classes of simple objects in $\Per_{H_\lambda}(V_\lambda)$ by $\{ \IC(\1_{C_j}) \tq j\in I\}$ then the pairing above becomes
\[
\langle [\pi_i],[\IC(\1_{C_j})]\rangle_{\lambda} =
\begin{cases}
    (-1)^{\dim C_i} & i = j,\\
    0 & \text{otherwise}.
\end{cases}
\]
If we change the scalars to $\CC$ throughout, then the pairing extends:
\begin{equation}\label{complex pairing}
\langle \cdot,\cdot \rangle: K_{\CC}\Rep^\text{fl}_\lambda(G) \times K_{\CC}\Per_{H_\lambda}(V_\lambda) \to \mathbb{C},
\end{equation}
where we set $K_{\CC}\Rep^\text{fl}_\lambda(G)\coloneqq \CC \otimes_{\mathbb{Z}} K\Rep^\text{fl}_\lambda(G)$ and $K_{\CC}\Per_{H_\lambda}(V_\lambda) \coloneqq \CC \otimes_{\mathbb{Z}} K\Per_{H_\lambda}(V_\lambda)$.

\subsection{Kazhdan-Lusztig Hypothesis}\label{section:KLH}

In this section we state the Kazhdan-Lusztig Hypothesis for $p$-adic general linear groups.

For every irreducible $\pi\in \Rep^\text{fl}_\lambda(G)$, let $\Delta(\pi)$ be the standard representation for $\pi$; thus, in particular, $\pi$ is the unique irreducible quotient of $\Delta(\pi)$.
For every $\pi_i$ and $\pi_j$ in $\Pi_\lambda(G)$, let $m_{i j}$ denoted the muliplicity of $\pi_i$ in $\Delta(\pi_j)$; thus, in the Grothendieck group $K\Rep^\text{fl}_\lambda(G)$,
\[
[\Delta(\pi_j)] = \sum_{i\in I} m_{i j} [\pi_i].
\]
Let $m_\lambda = (m_{i j})$ be the matrix of these entries.
It is possible to order $I$, and thus the representations appearing in $\Pi_\lambda(G)$, so that the matrix $m$ is lower triangular, with diagonal entries $1$; consequently, the matrix $m$ is invertible.
Notice that $m$ is the change of basis matrix for the vector space $K_\CC\Rep^{\op{fl}}_{\lambda}(G)$, from the basis $\{ [\Delta(\pi_i)] \tq i\in I\}$ to $\{ [\pi_j] \tq j\in I\}$. 


Return to the infinitesimal parameter $\lambda : W_F\to \dualgroup{G}$ and consider the abelian category $\Per_{H_\lambda}(V_\lambda)$ of $H_\lambda$-equivariant perverse sheaves on $V_\lambda$. 
Simple objects in this category are the intersection cohomology complexes $\IC (\1_{C_i})$. 
For each $H_\lambda$-orbit $C_j$ in $V_\lambda$, pick a base point $x_j\in C_j$ and let $c_{ij}$ be the Euler characteristic of the stalk of $\IC (\1_{C_j})[-\dim C_j]$ at ${x_i}$: 
\begin{eqnarray*}
{c}_{ij} 
= (-1)^{\dim C_j}\rank\left(\mathcal{H}^\bullet_{x_i}\IC(\1_{C_j}))\right).
\end{eqnarray*}
Set ${c_\lambda} = ({c}_{ij})$.


\begin{hypothesis}[$p$-adic analogue of the Kazhdan-Lusztig Hypothesis]\label{hypothesis}
In the Grothendieck group $K\Rep(G)$ the multiplicity of the irreducible representation $\pi_i$ in the standard representation $\Delta(\pi_j)$ is given by
\[
m_{i j} 
= 
(-1)^{\dim C_i}\rank(\mathcal{H}^\bullet_{x_j}\IC \left(\1_{C_i})\right).
\]
Equivalently, the change of basis matrix in $K\Rep^\text{fl}_\lambda(G)$ from standard representations to irreducible representations is computed by the Euler characteristics of stalks of simple objects in $\Per_{H_\lambda}(V_\lambda)$:
\[
m_\lambda = \,^t{c_\lambda}. 
\]
\end{hypothesis}

Hypothesis~\ref{hypothesis} was first articulated in \cite{zelevinskii1981p}. It also appears in \cite{ABV}*{Chapter 15} for real groups and in \cite{Vogan:Langlands}*{Section 8} for real and $p$-adic groups, though there are some sign errors in the latter.
For general linear groups, Hypothesis~\ref{hypothesis} is a folklore theorem, often attributed to \cite{CG} or \cite{Lusztig:Cuspidal2}. 
More recently, Hypothesis~\ref{hypothesis}, as it applies here, is also asserted in \cite{Solleveld:pKLH}*{Theorem E, (b) and (c)}.
In this paper we take Hypothesis~\ref{hypothesis} as given.

Using this notation, we revisit the matrix ${c}$ from Section~\ref{section:KLH} and write
\[
[\IC(\1_{C_j})] 
= 
(-1)^{\dim(C_j)} \sum_{i\in I} {c}_{ij} [\1_{C_i}^\natural]
\]
in $K\Per_{H_\lambda}(V_\lambda)$.
Thus, $c_\lambda$ is the change of basis matrix for the vector space $K_\mathbb{C}\Rep_\lambda(G)$, from the basis $\{ [\1_{C_i}^\natural] \tq i\in I \}$ to the basis $\{ [\IC(\1_{C_i})[-\dim C_i] \tq i\in I \}$. 
Likewise,
\[
[\1_{C_j}^\natural] 
= 
\sum_{i\in I} 
({c_\lambda}^{-1})_{ij} (-1)^{\dim(C_i)} [\IC(\1_{C_i})]
\]
in $K_\CC\Per_{H_\lambda}(V_\lambda)$.

Since standard representations form a basis for the Grothendieck group $K\Rep^\text{fl}_\lambda(G)$, it is natural to ask what objects of $K\Per_{H_\lambda}(V_\lambda)$ are dual to this basis, under the pairing of Equation~\eqref{eqn:pairing}.
In the lemma below we use Hypothesis~\ref{hypothesis} to show that standard sheaves in $K\Per_{H_\lambda}(V_\lambda)$ are dual to standard representation $[\Delta(\pi)]$ in $K\Rep^\text{fl}_\lambda(G)$.

\begin{lemma}\label{lemma:St}
For any $\pi\in \Pi_\lambda(G)$ and any $H_\lambda$-orbit $C$ in $V_\lambda$,
\[
\langle [\Delta(\pi)],[\1_{C}^\natural]\rangle_{\lambda} =
\begin{cases}
    1 & [\mathcal{P}(\pi)] = [\IC(\1_C)];\\
    0 & \text{otherwise}.
\end{cases}
\]
\end{lemma}

\begin{proof}
\begin{eqnarray*}
\langle [\Delta(\pi_j)], [\1_{C_i}^\natural] \rangle_\lambda
&=& 
\langle  
\sum_{k\in I} m_{ji} [\pi_l],
\sum_{l\in I } (-1)^{\dim(C_l)}({c)\lambda}^{-1})_{li} [\IC(\1_{C_l})]
\rangle_\lambda\\
&=& 
\sum_{k,l \in I} m_{ji}\ 
(-1)^{\dim(C_l)} ({c_\lambda}^{-1})_{li}
\langle 
[\pi_k],[\IC(\1_{C_l})] \rangle_\lambda\\
&=& 
\sum_{l\in I} m_{lj}\ 
(-1)^{\dim(C_l)} ({c_\lambda}^{-1})_{li}
(-1)^{\dim(C_l)}\\
&=& 
\sum_{l\in I} 
({c_\lambda}^{-1})_{li}\
m_{lj}
\\
&=&
\sum_{l\in I}  (\,^t{c_\lambda}^{-1})_{il}\ m_{lj}\\
&=&
(\,^t{c_\lambda}^{-1}\ m_\lambda)_{ij}.
\end{eqnarray*}
By Hypothesis~\ref{hypothesis}, $\,^t{c_\lambda}^{-1} = m_\lambda^{-1}$, so 
\begin{eqnarray*}
(\,^t{c_\lambda}^{-1}\ m_\lambda)_{ij}
&=&
(m_\lambda^{-1}\ m_\lambda )_{ij}\\
&=& 
\begin{cases}
    1 & i=j\\
    0 &i\ne j.
\end{cases}
\end{eqnarray*}
\end{proof}

\subsection{Alternate form of Vogan's conjecture on A-packets}\label{subsection:eta}

Now let $\psi : W''_F \to \dualgroup{G}$ be an Arthur parameter for $G$.
Let $\lambda \coloneqq \lambda_\psi : W_F\to \dualgroup{G}$ be its infinitesimal parameter, as defined in Section~\ref{subsection:notation}.
Based on \cite{CFMMX}*{Definition 2, \S 8.2}, define $\eta^{\Evs}_{\psi} \in K\Rep_\lambda(G)$ by
\[
\eta^{\Evs}_{\psi} \coloneqq 
(-1)^{d(\psi)}\sum_{\pi\in \Pi_{\lambda}(G)} (-1)^{d(\pi)} \rank \left(\Evs_\psi \mathcal{P}(\pi) \right)\ [\pi],
\]
where $d(\psi) \coloneqq \dim(C_\psi)$ and $d(\pi) = \dim(C_{\phi_\pi})$.
Recall that the classes $[\pi]$, as $\pi$ ranges over  $\Pi_\lambda(G)$, form a basis for $K\Rep^\text{fl}_\lambda(G)$.

\begin{proposition}\label{prop:eta}
For all $\mathcal{F}\in D_{H_\lambda}(V_\lambda)$,
\[
\langle \eta^{\Evs}_{\psi}, [\mathcal{F}] \rangle_\lambda
=
(-1)^{d(\psi)} \rank \left(\Evs_\psi \mathcal{F}\right).
\]
\end{proposition}

\begin{proof}
It is enough to prove the proposition in the case that $\mathcal{F}$ is a simple object in $\Per_{H_\lambda}(V_\lambda)$.
To see this, note that classes of simple objects in $\Per_{H_\lambda}(V_\lambda)$ form a basis for $K\Per_{H_\lambda}(V_\lambda)$, and since $\Per_{H_\lambda}(V_\lambda)$ is a heart of $\Der_{H_\lambda}(V_\lambda)$, the Grothendieck groups coincide: $K\Per_{H_\lambda}(V_\lambda)=K\Der_{H_\lambda}(V_\lambda)$.
Moreover, the functor $\Evs_\psi \mathcal{F}$ depends only on the class of $\mathcal{F}$ in $K\Der_{H_\lambda}(V_\lambda)$.

So now we assume $\mathcal{F}$ is a simple object in $\Per_{H_\lambda}(V_\lambda)$.
Recall that every simple object in this category takes the form $\mathcal{P}(\pi_i)$ for some $\pi_i\in \Pi_\lambda(G)$.
We now prove the Proposition for $\mathcal{F} = \mathcal{P}(\pi_i)$:
\[
\begin{array}{rlr}
    &\langle \eta^{\Evs}_\psi,[\mathcal{P}(\pi_i)]\rangle_{\lambda} 
        &
            \\
    &=
\langle (-1)^{d(\psi)} \mathop{\sum}\limits_{\pi\in \Pi_{\lambda}(G)} (-1)^{d(\pi)} \rank \left(\Evs_\psi \mathcal{P}(\pi) \right)[\pi_i], [\mathcal{P}(\pi)]\rangle_{\lambda}
        &
            \\
    &= (-1)^{d(\psi)} (-1)^{d(\pi_i)}
\rank \left(\Evs_\psi \mathcal{P}(\pi_i)\right) 
\langle [\pi_i], [\mathcal{P}(\pi_i)]\rangle_\lambda
        &
            \\
    &= (-1)^{d(\psi)} (-1)^{d(\pi_i)}
\rank \left(\Evs_\psi \mathcal{P}(\pi_i)\right) 
(-1)^{d(\pi_i)}
        & \text{by Equation~\ref{eqn:pairing}}
            \\
&= (-1)^{d(\psi)} 
\rank \left(\Evs_\psi \mathcal{P}(\pi_i)\right).
    &
\end{array}
\]
\end{proof}

Armed with these tools, we may now re-cast Vogan's conjecture on A-packets for $G(F)$ in the following form, which we will prove in Theorem~\ref{thm:main}:
\[
\eta^{\Evs}_\psi = \eta_\psi;
\]
or equivalently,
\[
\langle \eta^{\Evs}_\psi, [\mathcal{F}]\rangle_\lambda 
= 
\langle \eta_\psi, [\mathcal{F}]\rangle_\lambda,
\qquad \forall \mathcal{F} \in D_{H_\lambda}(V_\lambda).
\]

\section{Main result on the Levi subgroup}

The results of \cite{CR:irred} adapt to Levi subgroups of $G$, as we now explain.

\begin{proposition}\label{prop:M}
Let $\psi_M$ be an irreducible Arthur parameter for a Levi subgroup $M$ of $G$. Then
\[
\Pi^\ABV_{\psi_M}(M) = \Pi_{\psi_M}(M) = \{ \pi_{\psi_M}\} \qquad\text{and}\qquad
\eta_{\psi_M}^{\Evs} = \eta_{\psi_M} = [\pi_{\psi_M}].
\]
\end{proposition}

\begin{proof}
As in Section~\ref{section:eta}, we have
\[
\Pi^\ABV_{\psi_M}(M) 
\coloneqq
\{ \sigma\in \Pi_{\lambda_M}(M) \tq \Evs_{\psi_M} \mathcal{P}(\sigma) \ne 0 \},
\]
where $\lambda_M$ is the infinitesimal parameter of $\psi_M$.
In the definition of $\Pi_{\psi_M}(M)$, we see the simple object $\mathcal{P}(\sigma)$ in $\Per_{H_{\lambda_M}}(V_{\lambda_M})$ determined by $\sigma \in \Rep_{\lambda_M}(M)$ using the local Langlands correspondence. 
Recall $M = \GL_{n_1}\times\cdots\times\GL_{n_k}$; to simplify notation, set $M_i = \GL_{n_i}$ and let $\lambda_i : W_F \to \dualgroup{M_i}$ be the composition of $\lambda_M : W_F \to \dualgroup{M}$ with the projection $\dualgroup{M} \to \dualgroup{M_i}$.
Now factor $V_{\lambda_M}$:
\[
V_{\lambda_M} = V_{\lambda_1}\times\cdots\times V_{\lambda_k} ,
\]
where $V_{\lambda_i}$ is the moduli space of Langlands parameters for $M_i$ with infinitesimal parameter $\lambda_i$.
Likewise, set $H_{\lambda_i} \coloneqq Z_{\dualgroup{M_i}}(\lambda_i)$ so
$
H_{\lambda_M} = H_{\lambda_1}\times\cdots\times H_{\lambda_k}.
$
Then
\[
\Per_{H_{\lambda_M}}(V_{\lambda_M}) \cong \Per_{H_{\lambda_1}}(V_{\lambda_1})\boxtimes\cdots\boxtimes \Per_{H_{\lambda_k}}(V_{\lambda_k})
\]
(finite product of categories).
%
Since $\sigma$ is irreducible, $\sigma = \sigma_1\boxtimes\cdots\boxtimes\sigma_k$ where each $\sigma_i$ is irreducible.
Now the Langlands correspondence for $M$ attaches
\[
\mathcal{P}(\sigma) 
= 
\mathcal{P}(\sigma_1)\boxtimes\cdots\boxtimes\mathcal{P}(\sigma_k)
\]
to $\sigma$.

Finally, recall 
\[
\psi_M = \psi_1\boxtimes\cdots\boxtimes\psi_k,
\]
and write $x_{\psi_M} = (x_{\psi_1}, \ldots , x_{\psi_k}) \in V_{\lambda_M}$ for the corresponding elements in the moduli space; 
likewise write $y_{\psi_M} = (y_{\psi_1}, \ldots , y_{\psi_k}) \in V_{\lambda_M}^*$. 
By the theorem of Thom-Sebastiani \cite{Illusie} \cite{Massey},
\begin{eqnarray*}
\left( \RPsi_{y_{\psi_M}} \mathcal{P}(\sigma)\right)_{x_{\psi_M}}
&=& 
\left( \RPsi_{(y_{\psi_1}, \ldots , y_{\psi_k})} \mathcal{P}(\sigma_1)\boxtimes\cdots\boxtimes \mathcal{P}(\sigma_k)\right)_{(x_{\psi_1}, \ldots , x_{\psi_k})} \\
&=& 
\left( \RPsi_{y_{\psi_1}} \mathcal{P}(\sigma_1)\right)_{x_{\psi_1}}
\boxtimes\cdots\boxtimes\left( \RPsi_{y_{\psi_k}} \mathcal{P}(\sigma_k)\right)_{x_{\psi_k}} .
\end{eqnarray*}
Thus,
\[
\left( \RPsi_{y_{\psi_M}} \mathcal{P}(\sigma)\right)_{x_{\psi_M}}\ne 0 
\qquad\iff\qquad
\left( \RPsi_{y_{\psi_i}} \mathcal{P}(\sigma_i)\right)_{x_{\psi_i}} \ne 0,\ \forall i=1,\ldots ,k.
\]
Equivalently,
\[
\Evs_{\psi_M}\mathcal{P}(\sigma)\ne 0 
\qquad\iff\qquad
\Evs_{\psi_i}\mathcal{P}(\sigma_i)\ne 0,\ \forall i=1,\ldots ,k.
\]
By \cite{CR:irred}, 
\[
\Pi^\ABV_{\psi_i}(M_i) = \Pi_{\psi_i}(M_i) = \{ \pi_{\psi_i} \}, \ \forall i=1, \ldots, k.
\]
It now follows that
\[
\Pi^\ABV_{\psi_M}(M) = \Pi_{\psi_M}(M) = \{ \pi_{\psi_M}\}.
\]

Recall the definition:
\[
\eta^{\Evs}_{\psi_M} \coloneqq 
(-1)^{d(\psi_M)}\sum_{\sigma\in \Pi_{\lambda_M}(G)} (-1)^{d(\sigma)} \rank \left(\Evs_{\psi_M}\mathcal{P}(\sigma)\right)\ [\sigma],
\]
We have just seen that $\Pi^\ABV_{\psi_M}(M) = \{ \pi_{\psi_M}\}$.
Therefore,
\[
\eta^{\Evs}_{\psi_M} =
(-1)^{d(\psi_M)-d(\pi_{\psi_M})} \rank \left(\Evs_{\psi_M}\mathcal{P}(\pi_{\psi_M})\right)\  [\pi_{\psi_M}],
\]
By \cite{CFMMX}*{\S 8.2},
\[
\rank \left(\Evs_{\psi_M}\mathcal{P}(\pi_{\psi_M})\right)
=
1,
\]
so 
\[
\eta^{\Evs}_{\psi_M} = [\pi_{\psi_M}],
\]
as claimed.
\end{proof}

\section{Fixed-point Formula}\label{section:FPF}

The proof of the main result, Theorem~\ref{thm:main}, uses a fixed-point formula, explained in this section.

From Section~\ref{introduction}, recall that $V_{\lambda_M}$ is the subvariety of $V_{\lambda}$ fixed by $\Ad(s)$,
where $s\in \dualgroup{G}$ be a finite-order element such that  $\dualgroup{M} = Z_{\dualgroup{G}}(s)$: $V_{\lambda_M} = V_{\lambda}^s$.
Let 
\begin{equation}\label{eqn:e}
\varepsilon : V_{\lambda_M} \hookrightarrow V_\lambda
\end{equation}
be the obvious inclusion. 
Let $\varepsilon^* : \Der_{H_\lambda}(V_\lambda) \to \Der_{H_{\lambda_M}}(V_{\lambda_M})$ be the equivariant restriction functor of equivariant derived categories. 
We will also use the notation
\[
\mathcal{F}\vert_{V_{\lambda_M}} \coloneqq \varepsilon^*\mathcal{F}.
\]
While $\varepsilon^*$ is an exact functor, it does not take perverse sheaves to perverse sheaves.

\begin{lemma}\label{lemma:FPF}
Let $\psi$ and $\psi_M$ be as above.
For all $\mathcal{F}\in \Der_{H_\lambda}(V_\lambda)$,
\[
(-1)^{d(\psi)}\rank \left(\Evs_{\psi}\mathcal{F}\right) 
=
(-1)^{d(\psi_M)}
\rank \left(\Evs_{\psi_M} \mathcal{F}\vert_{V_{\lambda_M}}\right).
\]
\end{lemma}

\begin{proof}
By \cite{CFMMX}*{Proposition 7.8 and Definition 2}, the functor $\Evs_\psi$ is related to vanishing cycles by 
\[
\Evs_\psi \mathcal{F} 
=
(-1)^{d({\hat\psi})-\dim V_\lambda} \left( \RPsi_{y_\psi} [-1] \mathcal{F}\right)_{x_\psi},
\]
where 
\begin{itemize}
\item
$x_\psi$ is the point for $\phi_\psi$ in this moduli space $V_\lambda$; 
\item
$y_\psi$ is the point in the dual moduli space $V_\lambda^*$ matching the Langlands parameter $\phi_{\hat\psi}$ where ${\hat \psi}(w,x,y) \coloneqq \psi(w,y,x)$; and
\item
$\RPsi_{y_{\psi}}$ is Deligne's vanishing cycles functor.
\item
$d({\hat\psi})$ is the dimension of the $H_\lambda$-orbit of $y_\psi$ in $V_\lambda^*$.
\end{itemize}
Next, recall the relation between vanishing cycles and local Morse groups, as for example in \cite{GM:book}*{Part II, Chapter 6, Section 6.A.2}, so
\[
\left( \RPsi_{y_\psi} [-1] \mathcal{F}\right)_{x_\psi}
= 
A^\bullet_{y_\psi}(\mathcal{F}),
\]
where we view $y_\psi \in T^*_{C_\psi,x_\psi}(V_\lambda)$.
Here we use \cite{CFMMX}*{Proposition 6.1} to see that $(x_\psi,y_\psi)\in T^*_{H_\lambda}(V_\lambda)$ is regular, so $y_\psi$ is non-degenerate in the sense of Morse theory.
Combining these observations, it follows that
\[
\mathcal{H}^i \left( \Evs_\psi \mathcal{F}[\dim C_\psi] \right) = H^i(J,K;\mathcal{F}),
\]
where $(J,K)$ is normal Morse data corresponding to $y_\psi$ as a linear functional on $V_\lambda$, as in \cite{GM:book}*{Part II, Chapter 6, Section 6.A.1}.

Now recall that $M$ was chosen from $\psi$ precisely so that its image lies in $\dualgroup{M} = Z_{\dualgroup{G}}(s)$ and, consequently, $s\in Z_{\dualgroup{G}}(\psi)$. 
Recall also that for $G= \GL_n$, the group $Z_{\dualgroup{G}}(\psi)$ is connected, so $A_\psi = \pi_0(Z_{\dualgroup{G}}(\psi))$ is trivial.
This allows us to interpret $\rank \Evs_\psi \mathcal{F}$ as a Lefschetz number:
\[
\rank\left(\Evs_\psi \mathcal{F}\right)
=
\trace\left(s,\Evs_\psi \mathcal{F}\right)
=
(-1)^{d(\psi)} \sum_{i}(-1)^i \trace\left(s,H^i(J,K;\mathcal{F})\right).
\]
Arguing as in the proof of \cite{ABV}*{Theorem 25.8}, which makes essential use of \cite{GM:Lefschetz}, it now follows that
\[
\sum_{i}(-1)^i \trace\left(s,H^i(J,K;\mathcal{F})\right)
=
\sum_{i}(-1)^i \trace\left(s,H^i(J^s,K^s;\mathcal{F})\right);
\]
in other words,
\[
(-1)^{d(\psi)} \trace(s,\Evs_\psi \mathcal{F})
=
(-1)^{d(\psi_M)}\trace(s,\Evs_\psi \varepsilon^*\mathcal{F});
\]
equivalently,
\[
(-1)^{d(\psi)} \rank \Evs_\psi \mathcal{F}
=
(-1)^{d(\psi_M)} \rank \Evs_\psi \varepsilon^*\mathcal{F}.
\]
Here we have also used that, by construction, $\varepsilon(x_{\psi_M}) = x_{\psi}$ and likewise, $y_{\psi_M}$ maps to $y_{\psi}$ under $V_{\lambda_M}^* \hookrightarrow V_\lambda^*$, and by \cite{CFMMX}*{Proposition 6.1}, $(x_\psi,y_\psi)\in T^*_{H_\lambda}(V_\lambda)$ is regular, while the same result shows $(x_{\psi_M},y_{\psi_M})\in T^*_{H_{\lambda_M}}(V_{\lambda_M})$ is regular.
\end{proof}

\begin{proposition}\label{prop:FPF}
Let $M$ be any Levi subgroup of $G$ and let $\psi_M$ be any Arthur parameter for $M$; let $\psi$ be its lift to $G$. 
Let $\lambda_M$ (resp. $\lambda$) be the infinitesimal parameter of $\psi_M$ (resp. $\phi$).
Then
\[
\langle \eta^{\Evs}_{\psi} , [\mathcal{F}] \rangle_{\lambda}
=
\langle \eta^{\Evs}_{\psi_M} , [\mathcal{F}\vert_{V_{\lambda_M}}] \rangle_{\lambda_M},
\]
for every $\mathcal{F}\in \Per_{H_{\lambda}}(V_{\lambda})$.
\end{proposition}

\begin{proof}
For all $\mathcal{F}\in D_{H_\lambda}(V_\lambda)$,
\[
\begin{array}{rlr}
    {\langle \eta^{\Evs}_{\psi}, [\mathcal{F}] \rangle_{\lambda}}
    &= (-1)^{d(\psi)}
\rank\left(\Evs_\psi \mathcal{F}\right).
        & \text{by Proposition~\ref{prop:eta}}
            \\
    &= (-1)^{d(\psi_M)} \rank\left(\Evs_\psi \mathcal{F}\vert_{V_{\lambda_M}}\right)
        & \text{by Lemma~\ref{lemma:FPF}}
            \\
    &= \langle \eta^{\Evs}_{\psi_M}, [\mathcal{F}\vert_{V_{\lambda_M}}] \rangle_{\lambda_M}
        & \text{by Proposition~\ref{prop:eta}}.
\end{array}
\]
\end{proof}

\section{Endoscopic Lifting}\label{section:Lift}

Recall that $\varepsilon : V_{\lambda_M}\to V_{\lambda}$ is the natural inclusion of Section~\ref{section:FPF} and that $\varepsilon^* : D_{H_{\lambda}}(V_{\lambda}) \to D_{H_{\lambda_M}}(V_{\lambda_M})$ is the induced equivariant restriction functor. 

With reference to Section~\ref{subsection:Pairing}, let 
\begin{equation}
\Lift_M^G : K_\CC\Rep_{\lambda_M}(M) \to K_\CC\Rep_{\lambda}(G)
\end{equation}
be the linear transformation defined by 
\begin{equation}\label{equation:e_*}
\langle \Lift_M^G[\pi], [\mathcal{F}] \rangle_\lambda = 
\langle [\pi], [\varepsilon^*\mathcal{F}] \rangle_{\lambda_M}.
\end{equation} 
Following the nomenclature of \cite{ABV}*{Definition 26.18} for real groups, we refer to the linear transformation $\Lift_M^G$ as \emph{endoscopic lifting}.
In this section we show that $\Lift_M^G$ coincides with the linear transformation given by $\Ind_P^G$ on the level of Grothendieck groups.
In Section~\ref{sec:LSLift} we see that it also coincides with Langlands-Shelstad transfer.
\[
\begin{tikzcd}
K_\CC \Rep_\lambda(G)\times K_\CC \Per_{H_\lambda}(V_\lambda)\arrow[bend left]{dd}{\varepsilon^* \hskip4pt \text{(geometric restriction)}} \arrow{r}& \CC\\
\\
\arrow[bend left]{uu}{\text{(endoscopic lifting)}\hskip4pt \Lift_M^G} K_\CC \Rep_{\lambda_M}(M)\times K_\CC \Per_{H_{\lambda_M}}(V_{\lambda_M}) \arrow{r}& \CC
\end{tikzcd}
\]

\subsection{Endoscopic lifting of standard representations}\label{ssec:Lift}

\begin{proposition}\label{prop:Lift}
Let $\phi_M$ be any Langlands parameter for $M$ with infinitesimal parameter $\lambda_M$.
Let $\pi_{\phi_M}$ be the corresponding irreducible representation of $M(F)$. 
Then
\begin{equation*}
    \Lift_M^G\left([\Delta(\pi_{\phi_M})]\right)
    =
    [\Delta(\pi_\phi)],
\end{equation*}
where $\phi$ is the Langlands parameter for $G$ obtained by lifting $\phi_M$ via $\dualgroup{M} \hookrightarrow \dualgroup{G}$.
\end{proposition}

\begin{proof}
We prove the proposition by showing
\[
\langle 
\Lift_M^G\left([\Delta(\pi_{\phi_M})]\right), [\mathcal{F}] 
\rangle_{\lambda}
    =
\langle 
[\Delta(\pi_\phi)],
[\mathcal{F}] 
\rangle_{\lambda},
\]
for every $\mathcal{F}\in \Der_{H_\lambda}(V_\lambda)$.
To do this, it is sufficient to take $\mathcal{F} = \1_{C}^\natural$ and allow $C$ to range over $H_\lambda$-orbits in $V_\lambda$, since these sheaves provide a basis for the Grothendieck group.
Observe that
\begin{equation*}
    \varepsilon^*\left(\1_{C}^\sharp\right)
    =
    \1_{C\cap V_{\lambda_M}}^\sharp.
\end{equation*}
Consequently, in $K\Per_{H_\lambda}(V_\lambda)$, 
\begin{equation*}
[\varepsilon^*\left(\1_{C}^\sharp\right)]
    =
    \sum_{D}
    [\1_{D}^\sharp]
\end{equation*}
where the sum is taken over $H_{\lambda_M}$-orbits $D$ in $V_{\lambda_M}$ appearing in $C\cap V_{\lambda_M}$, or in other words, over all orbits $D$ in $V_{\lambda_M}$ whose saturation in $V_{\lambda}$ is $C$.
Now,
\begin{eqnarray*}
\langle 
\Lift_M^G\left([\Delta(\pi_{\phi_M})]\right), [\1_{C}^\natural] 
\rangle_{\lambda}
    &=&
\langle 
[\Delta(\pi_{\phi_M})], \varepsilon^*[\1_{C}^\natural] 
\rangle_{\lambda_M}
\\
&=&
\langle 
[\Delta(\pi_{\phi_M})], \sum_{D} [\1_{D}^\natural]
\rangle_{\lambda_M}
\\
&=& 
\sum_{D} 
\langle 
[\Delta(\pi_{\phi_M})],  [\1_{D}^\natural]
\rangle_{\lambda_M}.
\end{eqnarray*}
Now, by Lemma~\ref{lemma:St} adapted from $G$ to $M$, 
$
\langle 
[\Delta(\pi_{\phi_M})],  [\1_{D}^\natural]
\rangle_{\lambda_M}
$
is non-zero only whee $D$ is the $H_{\lambda_M}$-orbit of $\phi_M\in V_{\lambda_M}$, in which case the pairing gives the value $1$. 
Therefore,
\begin{eqnarray*}
\langle 
\Lift_M^G\left([\Delta(\pi_{\phi_M})]\right), [\1_{C}^\natural] 
\rangle_{\lambda}
&=&
\begin{cases}
1 & \phi\in C\\
0 & \phi\not\in C.
\end{cases}
\end{eqnarray*}
On the other hand, by Lemma~\ref{lemma:St},
\begin{eqnarray*}
\langle 
[\Delta(\pi_\phi)],
[\1_C^\natural] 
\rangle_{\lambda}
&=&
\begin{cases}
1 & \phi\in C\\
0 & \phi\not\in C.
\end{cases}
\end{eqnarray*}
It follows that
\[
\langle 
\Lift_M^G\left([\Delta(\pi_{\phi_M})]\right), [\1_C^\natural] 
\rangle_{\lambda}
    =
\langle 
[\Delta(\pi_\phi)],
[\1_C^\natural] 
\rangle_{\lambda},
\]
for every $H_\lambda$-orbit $C$ in $V_\lambda$, and therefore
\[
\langle 
\Lift_M^G\left([\Delta(\pi_{\phi_M})]\right), [\mathcal{F}] 
\rangle_{\lambda}
    =
\langle 
[\Delta(\pi_\phi)],
[\mathcal{F}] 
\rangle_{\lambda},
\]
for every $\mathcal{F}\in \Der_{H_\lambda}(V_\lambda)$.
Since the pairing is perfect, it follows that
\[
\Lift_M^G\left([\Delta(\pi_{\phi_M})]\right)
=
[\Delta(\pi_\phi)].
\]
\end{proof}

\subsection{Comparison with parabolic induction}\label{subsection:IndLift}

In this subsection, we make precise and prove the claim that parabolic induction of a standard representation is a standard representation. We then show that endoscopic lifting can be characterized by parabolic induction.

We use the theory and notation established in Section \ref{ssec: standard representations} for irreducible and standard representations. Let $\pi$ be an irreducible representation of $G(F)$. The standard representation $\Delta(\pi)$ of $\pi$ can be extracted in terms of multisegments using Zelevinsky theory. This can be seen directly from \cite{kudla1994local}*{Theorem 2.2.2}, but we also clarify this explicitly in the lemma below. 

\begin{lemma}\label{lemma: standard is second induction}
    Let $\pi$ be a smooth irreducible representation of $G=\GL_n(F)$ with multisegment $\alpha = \{\Delta_1, \ldots \Delta_k\}$ arranged so that the segments satisfy the "does not precede" condition. Then \[
    \Delta(\pi)\simeq \Ind_P^G(Q(\Delta_1)\otimes Q(\Delta_2)\otimes \cdots \otimes Q(\Delta_k)),
    \]
    where $P$ is the standard parabolic specified by the $\Delta_i$s.
\end{lemma}

\begin{proof}
    We use Kudla's expository work \cite{kudla1994local}, specifically the arguments in pages 372-374. Recall from \ref{ssec: standard representations} that $\pi$ is the unique irreducible quotient of $\Ind_P^G(Q(\Delta_1)\otimes Q(\Delta_2)\otimes \cdots \otimes Q(\Delta_k))$ where $\Delta_i$s are arranged so that they satisfy the "does not precede" condition. Each $Q(\Delta_i)$ is an essentially tempered representation, which means there is a $x_i \in \mathbb{R}$ so that $Q(\Delta_i)\simeq Q(\Delta'_i)(x_i)$ where $Q(\Delta'_i)$ is tempered. Thus we have
    \[
    \begin{array}{rl}
    &\Ind_P^G(Q(\Delta_1)\otimes Q(\Delta_2)\otimes \cdots \otimes Q(\Delta_k))\\ &\simeq \Ind_P^G(Q(\Delta'_1)(x_1)\otimes Q(\Delta'_2)(x_2)\otimes \cdots \otimes Q(\Delta'_k)(x_k)).
    \end{array}
    \]
    Since $Q(\Delta'_i)$s are square-integrable, none of the $\Delta'_i$s can be linked. Moreover, we must have $x_1 \geq x_2 \ldots \geq x_k$ as the $\Delta_i$s satisfy the "does not precede" condition. If $x_i=x_{i+1}$, then we can replace $Q(\Delta'_i)(x_i)\otimes Q(\Delta'_{i+1})(x_{i+1})$ with $Q(\Delta'_i, \Delta'_{i+1})(x_i)$, which is equal to the full induced representation, and an irreducible tempered representation twisted by $x_i$. Thus, we obtain a sequence $x_1 > \cdots > x_{k'}$ and tempered representations $\tau_1=Q(\alpha_1),\ldots,\tau_{k'}=Q(\alpha_{k'})$ where $\alpha_i$s partition the set $\{\Delta'_1,\Delta'_2,\ldots, \Delta'_{k}\}$. This gives us,
    \begin{equation}\label{standard is second induction}
        \Ind_P^G(Q(\Delta_1)\otimes Q(\Delta_2)\otimes \cdots \otimes Q(\Delta_k))) \simeq \Ind_{P'}^G(\tau_1(x_1)\otimes \cdots \otimes \tau_{k'}(x_{k'})).\end{equation}
    Here $P'=M'N'$ is the standard parabolic subgroup specified by the $\alpha_i$s. By observing that $\tau_1\otimes \cdots \otimes \tau_{k'}$ is a tempered representation of $M'$ and $x_1 > \cdots > x_{k'}$ specifies a $P'$-positive unramified character of $M'$, we see that the representations in \eqref{standard is second induction} are standard representations. The result now follows by using the that $\pi$ is the unique irreducible quotient of $\Ind_P^G(Q(\Delta_1)\otimes Q(\Delta_2)\otimes \cdots \otimes Q(\Delta_k))$.
\end{proof}

The upshot of this lemma is that we can talk about standard representations purely in terms of multisegments, which makes it easy to pin down the standard representation obtained after parabolic induction. We prove the implicit claim in this statement below.

\begin{proposition}\label{prop:Ind standard}
Let $M=\GL_{m_1}\times \GL_{m_2} \cdots \times \GL_{m_k}$ be a Levi subgroup of $G$.  Let $P$ denote the standard parabolic of $G$ with Levi component $M$. Then, for any $\pi_M \in \Pi_{\lambda_M}(M)$, $[\Ind_P^{G}\left(\Delta(\pi_M)\right)]$ is the image of a standard representation of $G$ in $K_{\CC}\Rep_{\lambda}(G)$. Moreover, $\Ind_P^{G}\left(\Delta(\pi_M)\right)$ has a unique composition factor $\pi$ so that \[[\Delta(\pi)]=[\Ind_P^{G}\left(\Delta(\pi_M)\right)].\]

\end{proposition}

\begin{proof}

A representation $\pi_M \in \Pi_{\lambda_M}(M)$ can be written as $\pi_1\otimes \cdots \otimes \pi_k$, where $\pi_i \in \Pi_{\lambda_i}(\GL_{m_i})$ for $1\leq i \leq k$. Moreover, $\Delta(\pi_M)=\Delta(\pi_1)\otimes \cdots \otimes \Delta(\pi_k)$. It suffices to prove this proposition for the case $k=2$. Each $\pi_i$ has the associated data of $\Delta^i_1, \ldots, \Delta^i_{k_i}$ and $x^i_i>x^i_2>\cdots >x^i_{k_i}$ where $Q(\Delta^i_j)$s are irreducible tempered representations. Then, from Lemma \ref{lemma: standard is second induction}, $\Delta(\pi_i)=\Ind_{P_i}^{\GL_{m_i}}(Q(\Delta^i_1)(x^i_1)\otimes \cdots \otimes Q(\Delta^i_{k_i})(x^i_{k_i}))$, where $P_i$ is specified by the $Q(\Delta^i_j)$s. Thus, we have 
\begin{align*}
 &\Delta(\pi_1)\otimes \Delta(\pi_2) \\
 &= \Ind_{P_1}^{\GL_{m_1}}(Q(\Delta^1_1)(x^1_1)\otimes \cdots \otimes Q(\Delta^1_{k_1})(x^1_{k_1})) \otimes \Ind_{P_2}^{\GL_{m_2}}(Q(\Delta^2_1)(x^2_1)\otimes \cdots \otimes Q(\Delta^2_{k_2})(x^2_{k_2})).
\end{align*}
Let $P$ be the standard parabolic subgroup of $G$ with Levi component $\GL_{m_1}\times \GL_{m_2}$. Applying the exact functor, $\Ind_P^G$ throughout, we get 
\begin{align*}
 &\Ind_P^G(\Delta(\pi_1)\otimes \Delta(\pi_2)) \\
 &\simeq \Ind_{P_{12}}^{G}(Q(\Delta^1_1)(x^1_1)\otimes \cdots \otimes Q(\Delta^1_{k_1})(x^1_{k_1}) \otimes Q(\Delta^2_1)(x^2_1)\otimes \cdots \otimes Q(\Delta^2_{k_2})(x^2_{k_2})).
\end{align*}
Here $P_{12} \subset P$ is the standard parabolic subgroup specified by the $Q(\Delta^i_j)$s and the identification follows from transitivity of induction. We rearrange the $Q(\Delta^i_j)(x^i_j)$s so that the $x^i_j$ are decreasing, and whenever two consecutive $x^i_j$s are equal, we may replace $Q(\Delta^i_j)(x^i_j)\otimes Q(\Delta^{i'}_{j'})(x^i_j)$ by $Q(\Delta^i_j, \Delta^{i'}_{j'})(x^i_j)$, which is the full induced representation and also an irreducible tempered representation twisted by $x^i_j$. This rearrangement does not affect the representative of $\Ind_{P_{12}}^{\GL_n}(Q(\Delta^1_1)(x^1_1)\otimes \cdots \otimes Q(\Delta^1_{k_1})(x^1_{k_1}) \otimes Q(\Delta^2_1)(x^2_1)\otimes \cdots \otimes Q(\Delta^2_{k_2})(x^2_{k_2}))$ in $K_{\CC}\Rep_{\lambda}(G)$; this follows from \cite{Z2}*{Thoerem 1.2}. Thus, we obtain a decreasing sequence $y_1,\ldots,y_{l}$ from the $x^i_j$s and multisets $\alpha_1, \ldots, \alpha_{l}$ which partition the $\Delta^i_j$s. Setting $\tau_i=Q(\alpha_i)$, we may write
\[[\Ind_P^{\GL_m(F)}(\Delta(\pi_1)\otimes \Delta(\pi_2))]=[\Ind_{P'_{12}}^{\GL_n(F)}(\tau_1(y_1)\otimes \cdots \otimes \tau_l(y_l)].\]

Here $P'_{12}=M'N'$ is the standard parabolic subgroup specified by the $\alpha_i$s. Now  $\tau_1\otimes \cdots \otimes \tau_2$ is a tempered representation of $M'$ and $y_1 > \cdots > y_l$ specifies a $P'_{12}$-positive unramified character of $M'$. This shows that the representation in the above equation is a standard representation of $G$. 

We now show a unique choice of $\pi$ so that $[\Delta(\pi)]=[\Ind_P^{G}(\Delta(\pi_1)\otimes \Delta(\pi_2))]$. This is completely determined by the multisegment data of the $\tau_i(y_i)$s as we explain below. For a segment $\Delta=[\rho(b),\rho(e)]$, set the notation $\Delta(x)=[\rho(b+x),\rho(e+x)]$. For a multisegment $\beta=\{\Delta_1,...,\Delta_s\}$, set $\beta(x)=\{\Delta_1(x), \cdots, \Delta_s(x)\}$. With this in mind, write $\alpha=\alpha_1(y_1) \sqcup \alpha_2(y_2)\sqcup \cdots \sqcup \alpha_l(y_l)$ where $\alpha_i$s and $y_i$s were determined above. If we write this disjoint union like a concatenation, \textit{i.e.,} preserve the order of $\alpha_i$s and the segments within them, then this multisegment satisfies the "does not precede" condition due to the procedure carried out above. This $\alpha$ corresponds to a unique irreducible representation $Q(\alpha)$ obtained from Langlands classification via multisegments, which is the unique irreducible quotient of $\Ind_{P'_{12}}^{G}(\tau_1(y_1)\otimes \cdots \otimes \tau_l(y_l))$. Setting $\pi=Q(\alpha)$, $\Delta(\pi)=\Ind_{P'_{12}}^{\GL_n(F)}(\tau_1(y_1)\otimes \cdots \otimes \tau_l(y_l))$. Thus, we have
\[[\Ind_P^{G}(\Delta(\pi_1)\otimes \Delta(\pi_2))]=[\Delta(\pi)]\]
in $K_{\CC}\Rep_{\lambda}(G)$.

\end{proof}
\begin{remark}\label{remark: standard identification multisegments}
    In the proof above, observe that $\alpha$ is given by the disjoint union of the multisegments of $\pi_1$ and $\pi_2$ \textit{after} appropriate rearrangement to satisfy the "does not precede" condition. Thus, it is easy to see how the multisegment data of the $\pi_i$s completely determines the representation $\Delta(\pi)$. This procedure generalizes to $k$ representations: If $\pi_i$ corresponds to $\alpha_i$, set $\alpha=\sqcup_i \alpha_i$ rearranged so that the segments satisfy the "does not precede" condition. Then $\pi=Q(\alpha)$ is the uniquely determined Langlands quotient of $\Ind_P^G(\otimes_{\Delta \in \alpha}Q(\Delta))$ so that 
    \[[\Delta(\pi)]=[\Ind_P^{G}(\Delta(\pi_1)\otimes \cdots \otimes \Delta(\pi_k))]=[\Ind_P^{G}(\Delta(\pi_M)].\]
where $\Delta(\pi_M) = \Delta(\pi_1)\otimes \cdots \otimes \Delta(\pi_k)$. Once again, we are able to rearrange the segments of $\alpha$ because we are working with full induced representations in the Grothendiek group, and can therefore invoke \cite{Z2}*{Theorem 1.2}, which assets that this rearrangement should not change the representative in the Grothendiek group.
\end{remark}
This property of induction coincides with endoscopic lifting. Thus, we see that endoscopic lifting is characterized by parabolic induction in this case. 
 \begin{proposition}\label{prop: induction and langlands functorialty}
     Let $M \simeq \GL_{m_1} \times \GL_{m_2}\times \cdots \times \GL_{m_k}$ be a Levi subgroup of $G$. Let $P=MN$ be the standard parabolic subgroup with Levi component $M$. Then, for any $[\pi] \in K_{\CC}\Rep_{\lambda}(M)$, 
     \[\Lift_M^G([\pi])=[\Ind_P^{G}(\pi)].\]
 \end{proposition}
 \begin{proof}
    We show that $\Lift_M^G$ and $\Ind_P^{G}$ have the same image on the basis consisting of standard representations of $K_{\CC}\Rep_{\lambda}(M)$. Let $\phi$ be a langlands parameter for $G$ with infinitesimal parameter $\lambda$, both of which factor through $M$. We denote the Langlands and infinitesimal parameters for $M$ using $\phi_M$ and $\lambda_M$, respectively. Now $\phi$ determines a multisegment $\alpha$ whose segments are arranged in an order so that it satisfies the "does not precede" condition, which determines a smooth irreducible representation, $\pi_{\phi}$, of $G$. The parameter $\phi_M$ for $M$ determines multisegments $\alpha_i$ for $1\leq i \leq k$, which determine smooth irreducible representations $\pi_{\phi_i}$ of $\GL_{m_i}$ so that 
     \[\pi_{\phi_M} = \pi_{\phi_1}\otimes \cdots \otimes \pi_{\phi_k},\]
     is interpreted as an external tensor product and thus a representation of $M$. Since $\phi$ factors through $M$ via $\phi_M$, $\alpha = \sqcup_{i=1}^k\alpha_i$, upto rearrangement of segments. Using Remark \ref{remark: standard identification multisegments}, we have that 
     \[[\Ind_P^{G}(\Delta(\pi_{\phi_M}))]=[\Delta(\pi_{\phi})].\]
     However, from Lemma \ref{prop:Lift}, we have that 
     \[\Lift_M^G([\Delta(\pi_{\phi_M})])=[\Delta(\pi_{\phi})].\]
     Since the two maps agree on the basis of standard representations, they are the same. 
 \end{proof}
Finally, we show that endoscopic lifting identifies the $A$-packet of the Levi with the $A$-packet of $G$. 
 \begin{proposition}\label{prop: Lift on irreducibles of A type}
 For every $\mathcal{F}\in D_{H_{\lambda}}(V_{\lambda})$,
\[
\langle \eta_{\psi_M},[\mathcal{F}\vert_{V_{\lambda_M}}]\rangle_{\lambda_M} 
=
\langle \eta_{\psi},[\mathcal{F}]\rangle_{\lambda}; 
\]
equivalently,
\begin{equation*}
    \Lift_M^G[\pi_{\psi_M}]
    =
    [\pi_\psi].
\end{equation*}
\end{proposition}
\begin{proof}
    Let $P$ be the standard parabolic subgroup of $G$ with Levi component $M$. The representation $\pi_{\psi_M}$ is a product of unitary Speh representations. We know that $\Ind_P^G(\pi_{\psi_M})$ is an irreducible representation of $G$, see \cite{Atobe}*{Section 2.4}, for example. By matching mutisegments of $\phi_{\psi_M}$ and $\phi_{\psi}$, we have \begin{equation}\label{eq: induction is irreducible on A-par}
     [\Ind_P^G(\pi_{\psi_M})]=[\pi_{\psi}]. 
    \end{equation}
Recall that $\eta_{\psi}=[\pi_{\psi}]$, as the A-packet for $\psi$ is a singleton. We know from \ref{prop:M} that $\Pi_{\psi_M}(M)=\{\pi_{\psi_M}\}$, which gave us $\eta_{\psi_M}=[\pi_{\psi_M}]$. Now we have
\[
    \begin{array}{rlr}
\langle \eta_{\psi_M}, [\mathcal{F}\vert_{V_{\lambda_M}}] \rangle_{\lambda_M}
    &= \langle [\pi_{\psi_M}], \varepsilon^*[\mathcal{F}] \rangle_{\lambda_M}
        & \text{\text{$\Pi_{\psi_M}(M) = \{ \pi_{\psi_M} \}$}},
            \\
    &= \langle \Lift_M^G[\pi_{\psi_M}], [\mathcal{F}] \rangle_{\lambda_M}
        & \text{Definition \ref{equation:e_*}},
            \\        
    &= \langle [\Ind_P^G \left(\pi_{\psi_M}\right)], [\mathcal{F}] \rangle_{\lambda_M}
        & \text{Proposition~\ref{prop: induction and langlands functorialty}},
            \\  
    &= \langle [\pi_{\psi}], [\mathcal{F}] \rangle_{\lambda}
        & \text{by \eqref{eq: induction is irreducible on A-par}},
            \\     
    &= \langle \eta_{\psi}, [\mathcal{F}] \rangle_{\lambda}
        & \text{$\Pi_{\psi}(G) = \{ \pi_\psi \}$}.        
\end{array}
\]
\end{proof}

\section{Main result}

\begin{theorem}[Vogan's conjecture for A-packets for $p$-adic general linear groups]\label{thm:main}
For every $p$-adic field $F$, every positive integer $n$ and every Arthur parameter $\psi$ for $\GL_n(F)$, the A-packet for $\psi$ coincides with the ABV-packet for the Langlands parameter $\phi_\psi$, and the virtual representation attached to this packet agrees with Arthur's:
\[
\Pi^\ABV_{\phi_\psi}(G) = \Pi_\psi(G) = \{ \pi_\psi\},
\qquad\text{and}\qquad
\eta^{\Evs}_\psi = \eta_\psi = [\pi_\psi].
\]
\end{theorem}

\begin{proof}
The proof is obtained by the following diagram, which we call the "endoscopy square", in which $\mathcal{F}\in \Der_{H_\lambda}(V_\lambda)$ is arbitrary.
\[
\begin{tikzcd}
    {\langle \eta^{\Evs}_{\psi}, [\mathcal{F}] \rangle_{\lambda}}
    \arrow[equal]{d}[swap]{\text{Fixed-point formula \hskip3pt}}
    \arrow[dashed,equal]{rrrr}{\text{Vogan's conjecture for $\psi$}}
    &&&& 
    {\langle \eta_{\psi}, [\mathcal{F}] \rangle_{\lambda}}
    \\
    \langle \eta^{\Evs}_{\psi_M}, [\mathcal{F}\vert_{V_{\lambda_M}}] \rangle_{\lambda_M}
    \arrow[equal]{rrrr}[swap]{\text{Vogan's conjecture for $\psi_M$}}
    &&&& 
    \langle \eta_{\psi_M}, [\mathcal{F}\vert_{V_{\lambda_M}}] \rangle_{\lambda_M} 
    \arrow[equal]{u}[swap]{\text{\hskip3pt Endoscopic lifting}}
\end{tikzcd}
\]
We establish the equality across the top by verifying the equality on the other three sides.
The left-hand side of the endoscopy square is a consequence of Proposition~\ref{prop:FPF}.
The equality on the bottom of the endoscopy square is direct consequence of Proposition~\ref{prop:M}; we remark that this result makes use of the main result from \cite{CR:irred}.
The right-hand side of the endoscopy square is Proposition~\ref{prop: Lift on irreducibles of A type}.
We may now conclude
\[
\langle \eta^{\Evs}_{\psi},[\mathcal{F}]\rangle_\lambda 
=
\langle \eta_{\psi},[\mathcal{F}]\rangle_\lambda ,
\]
for every $\mathcal{F}\in D_{H_\lambda}(V_\lambda)$. 
Since the pairing above is non-degenerate, if follows that
\[
\eta^{\Evs}_{\psi} = \eta_{\psi}.
\]
Since $\eta_\psi = [\pi_\psi]$, it follows that
\[
\Pi^\ABV_\psi(G) = \{ \pi_\psi \} = \Pi_\psi(G).
\]
This concludes the proof of Vogan's conjecture for A-packets of general linear groups.
\end{proof}

\section{Langlands-Shelstad transfer and endoscopic lifting}\label{sec:LSLift}

In this section we show that the endoscopic lifting $\varepsilon^*$ from Section{} coincides with Langlands-Shelstad transfer from the Levi subgroup $M$ to the general linear group $G$. 
This result does not play a role in the proof of the main result, Theorem~\ref{thm:main}, so it is offered here as a remark.

Recall that Langlands-Shelstad transfer is defined first on Schwartz functions. Since we are considering the case $G=\GL_n$ and $M = \GL_{m_1}\times\cdots\times\GL_{m_k}$ for $n=m_1 + \cdots + m_k$, there is no need to mention stability in this context. 
The geometric transfer coefficients $\Delta(\gamma,\delta)$ are very simple in this case: functions $f\in C^\infty_c(G(F))$ and $f^M\in C^\infty_c(M(F))$ are said to match if 
\[
\mathcal{O}^M_\gamma(f^M) 
= 
 \Delta(\gamma,\delta)\ \mathcal{O}^G_\delta(f), 
\]
for regular semisimple $\gamma\in M(F)$ and $\delta\in G(F)$, where $\Delta(\gamma,\delta)=0$ unless $\gamma$ and $\delta$ have the same characteristic polynomials, in which case
\[
\Delta(\gamma,\delta) = \lvert {\det}_{\mathfrak{g}/\mathfrak{m}}\left(1-\Ad(\gamma) \right)\lvert_F.
\]
See, for example, \cite{Clozel}*{\S 2}. The factor $\Delta(\gamma, \delta)$ agrees with the Langlands-Shelstad transfer factor for this case, which equals $\Delta_{IV}$, as defined in \cite{langlands-shelstad}*{Section 3}.
In fact, in the case at hand, Langlands-Shelstad transfer is given by the linear transformation
\begin{align*}
  C^\infty_c(G(F)) &\to  C^\infty_c(M(F))\\
  f &\mapsto f^M
\end{align*}
defined by
\[
f^M(m) = \delta_{P}^{1/2}(m) \int_{K} \int_{N(F)} f(kmuk^{-1})\, du\, dk
\]
where $K$ is the maximal compact subgroup of $G$, $N$ is the unipotent radical of the standard parabolic $P$ with Levi component $M$ and $\delta_P$ is the modulus quasicharacter for $P(F)$.

Recall also that distributions $D$ on $G(F)$ and $D^M$ on $M(F)$ are related by Langlands-Shelstad transfer if
\[
D^M(f^M) = D(f),
\]
for all $f\in C^\infty_c(G(F))$.

We recall the distribution character $\Theta_\pi$ attached to an admissible representation $(\pi,V)$ of $G$. For any $f \in C_c^{\infty}(G)$, the linear operator 
\[\pi(f)v=\int_G f(g)\pi(g)vdg\]is of finite rank, by admissibility of $\pi$. Therefore, it has a well-defined trace \[\Theta_{\pi}(f) := \op{tr}\pi(f).\]
Furthermore, Harish-Chandra's work gives us a locally integrable function $\theta_{\pi}$ on $G$ so that $\Theta_{\pi}$ is written in terms of $\theta_{\pi}$. 
\begin{equation}\label{big theta small theta}
\Theta_{\pi}(f)=\int_{G}f(g)\theta_{\pi}(g) dg.
\end{equation}
We note here that the above is true for reductive $p$-adic groups, and not just general linear groups. We also note that $\Theta_{\pi_1}=\Theta_{\pi_2}$ if $[\pi_1]=[\pi_2]$ in $K_{\CC}\Rep_{\lambda}(G)$.

\begin{lemma}\label{lemma: LS matches standards with standards}
Let $\pi_M$ be an irreducible admissible representation of $M(F)$. 
Let $\phi_M$ be its Langlands parameter of $\pi_M$; let $\phi$ be the lift of $\phi_M$ to $G$ and let $\pi$ be the irreducible admissible representation of $G(F)$ matching $\pi$ under the Langlands correspondence. 
Recall that $\Delta(\pi)$ (resp. $\Delta(\pi_M)$) denotes the standard representation for $\pi$ (resp. $\pi_M$).
Then Langlands-Shelstad transfer matches standard representations with standard representations: 
\begin{equation*}
    \Theta_{\Delta(\pi_M)}(f^M) 
    = 
    \Theta_{\Delta(\pi)}(f),
\end{equation*}
for all $f\in C^\infty_c(G(F))$.
\end{lemma}

\begin{proof}

We show this by direct calculation. Recall that $\pi_M \in \Pi_{\lambda_M}(M)$ is matched with a $\pi \in \Pi_{\lambda}(G)$ by matching their multisegments, in the sense of Remark \ref{remark: standard identification multisegments}. Set $\tau=\Ind_P^G(\pi_M)$ where $P$ is the standard parabolic subgroup of $G$ with Levi subgroup $M$. Recall that $K$ is the maximal compact subgroup of $G$, and we have $G(F)=K\, M(F)\, N(F)$. After making all the appropriate choices for Haar measures, we have
\[
    \begin{array}{rlr}
&\Theta_{\Delta(\pi_M)}(f^M)\\
    &= \int_{M(F)} f^M(m)\theta_{\Delta(\pi_M)}(m)\, dm
        & \text{by \eqref{big theta small theta}},
            \\
    &= \int_{M(F)} \delta_{P}^{1/2}(m) \int_{K} \int_{N(F)} f(kmuk^{-1})\, du\, dk \, \theta_{\Delta(\pi_M)}(m) \, dm
        & \text{\hskip10pt definition of $f^M$},
            \\        
    &= \delta_{P}^{1/2}(m) \int_K \int_{N(F)} \int_{M(F)} \theta_{\Delta(\pi_M)}(m)f(kmuk^{-1})\,  dm \, dn \, dk 
        & \text{Fubini-Tonelli},
            \\  
    &= \Theta_{\tau}(f),        
\end{array}
\]
where the last equality follows from \cite{vanDijk}*{Theorem 2} paired with the remark at the end of Section 5 in \textit{loc. cit}. We know from Proposition \ref{prop:Ind standard} that $[\tau]=[\Ind_P^G(\Delta(\pi_M))]=[\Delta(\pi)]$. In this sense, the transfer of functions $f \mapsto f^M$ matches distribution characters of standard representations at the level of Grothendiek groups.
\end{proof}

Passing from distributions build from irreducible representations to the Grothendieck group of these representations, Langlands-Shelstad transfer defines a linear transformation 
\[
\op{LS} : K_\CC \Rep_{\lambda_M}(M) \to K_\CC \Rep_{\lambda}(G).
\]
This linear transformation sends standard representations to standard representations exactly as in Proposition \ref{prop:Ind standard} and Remark \ref{remark: standard identification multisegments}. Thus, this is another way to characterize endoscopic transfer, analogous to Proposition \ref{prop: induction and langlands functorialty}.
\begin{proposition}\label{prop:Langlands shelstad endoscopic transfer}
     Let $M \simeq \GL_{m_1} \times \GL_{m_2}\times \cdots \times \GL_{m_k}$ be a Levi subgroup of $G$. Let $P$ be the standard parabolic subgroup with Levi component $M$. Then, for any $[\pi] \in K_{\CC}\Rep_{\lambda}(M)$, 
     \[\Lift_M^G([\pi])=\LS([\pi]).\]
 \end{proposition}

Finally, we show that LS lifts A-packets from $M$ to A-packets of $G$. This follows immediately once we recall that $\Ind_P^G(\pi_{\psi_M})\simeq \pi_{\psi}$ from the proof of Proposition \ref{prop: Lift on irreducibles of A type}. Now, we may re-purpose the proof of Lemma \ref{lemma: LS matches standards with standards}, to assert the following.
 \begin{proposition}\label{LS matches A packets with A packets}
Langlands-Shelstad transfer matches $\pi_{\psi_M}$ with $\pi_{\psi}$, in the sense that 
\begin{equation*}
    \Theta_{\pi_{\psi_M}}(f^M)
    =
    \Theta_{\pi_{\psi}}(f)
\end{equation*}
for any $f \in C_c^{\infty}(G(F))$.
\end{proposition}


\begin{remark}
While it may appear that we are repackaging the results of Section \ref{ssec:Lift} in a slightly different language, our purpose here is to demonstrate that Langlands-Shelstad transfer will potentially give us the results in more general settings that we obtain from parabolic induction for general linear groups. In particular, we expect Langlands-Shelstad transfer to match standard representations with standard representations at the level of Grothendiek groups, just as parabolic induction does. However, parabolic induction of a representation from the $A$-packet of a Levi subgroup (or more generally, an endoscopic group) may not be irreducible, and is therefore not a good candidate for lifting A-packets of the Levi subgroup to A-packets of the group. In future work, we study Vogan's conjecture for a classical group $G$. Suppose $H$ is an endoscopic group of $G$. We propose an independent study of Langlands-Shelstad transfer to obtain the image $\op{Lift}_H^G([\pi_{\psi_H}])$. 
\end{remark}
\section{Examples}\label{sec:examples}

In this section, we provide examples to supplement the theory developed in the paper. Although this paper generalizes the situation from a simple Arthur parameter to an arbitrary parameter, we begin with a simple parameter and then move on to sums of simple parameters. 

\begin{example}[Steinberg $\GL_2$]\label{example:Steinberg GL2}
In this example, we work with a simple Arthur parameter, calculate the spectral and geometric multiplicity matrices, and demonstrate Hypothesis \ref{hypothesis} in this case. For $G=\GL_2$ over $F$, consider the Arthur parameter
$\psi : W''_F \to \dualgroup{G}$ defined by
\[\psi(w,x,y)=\Sym^1(x).\]Then, $\phi_{\psi}(w,x)=\Sym^1(x)$ and $\lambda(w) = \diag(|w|^{1/2}, |w|^{-1/2})$.

We start with the spectral side: $\Rep^\text{fl}_\lambda(G)$ contains exactly two irreducible representations - the trivial representation $\pi_0$ and the Steinberg representation $\pi_1$. The Steinberg representation is its own standard representation because it is tempered, so  $\Delta(\pi_1) = \pi_1$. 
The standard representation for the trivial representation $\pi_0$ is $\Delta(\pi_0) = \Ind_B^G(\chi)$, where $\chi(\diag(t_1, t_2)) = |t_1|^{1/2}|t_2|^{-1/2}$ and $B$ is the standard Borel subgroup of $\GL_2$. 
From the short exact sequence
\[
\begin{tikzcd}
0 \arrow{r} & \pi_0 \arrow{r} & \Ind_B^G(\chi) \arrow{r} & \pi_1 \arrow{r} & 0
\end{tikzcd}
\]
we see that $\pi_0$ and $\pi_1$ both appear in $\Delta(\pi_0)$ with multiplicity $1$, so
\begin{eqnarray*}\, 
[\Delta(\pi_0)] &=& [\pi_0] + [\pi_1], \text{ and}\\ \,
[\Delta(\pi_1)] &=& [\pi_1]
\end{eqnarray*}
in $K\Rep^\text{fl}_\lambda(G)$. Thus, we have
\[
m = 
\begin{bmatrix}
1 & 0\\
1 & 1
\end{bmatrix}.
\]

Now we describe the geometry. 
\[ V_\lambda= \left\{\begin{pmatrix}0&x\\0&0\end{pmatrix}: \text{ }x\in \CC \right\}\simeq\mathbb{A}^1_{\CC},\] 

and \[H_\lambda = \GL_1(\CC)\times \GL_1(\CC),\] with action $s\cdot x = s_1 x s_2^{-1}$, where $s=(s_1,s_2)$.
The two $H_\lambda$-orbits in $V_\lambda$ are $C_0 = \{ 0 \}$ and $C_1 = \{ x\in \mathbb{A}^1 \ : \ x\ne 0 \}$, and simple objects in $\Per_{H_\lambda}(V_\lambda)$ are $\IC(\1_{C_0})$ and $\IC(\1_{C_1})$. 
In order to compute the matrix $c$, we pick base points $x_0 \in C_0$ and $x_1\in C_1$ and compute the stalks: The sheaf complex $\IC(\1_{C_0})$ is the skyscraper sheaf $\1_{C_0}^\natural$ at $C_0$ in degree $0$ while $\IC(\1_{C_1})$ is the constant sheaf on $V_\lambda$ shifted by $1$, $\1_{V}[1]$. 
It follows that 
\begin{eqnarray*}\,
[\IC(\1_{C_0})] &=& [\1_{C_0}^\natural] \\\, 
(-1)[\IC(\1_{C_1})] &=& [\1_{C_0}^\natural] + [\1_{C_1}^\natural],
\end{eqnarray*}
in $K\Per_{H_\lambda}(V_\lambda)$, so
\[
c = 
\begin{bmatrix}
1 & 1\\
0 & 1
\end{bmatrix}.
\]
It is now clear that $m = \,^tc$, as predicted by Hypothesis~\ref{hypothesis}.
\end{example}

\begin{example}\label{example:braden gl4}
We now consider an Arthur parameter of small dimension which is not simple. We show the computation of the geometric and spectral multiplicity matrices $c$ and $m$. We will also see in the next two examples that the Vogan variety for this parameter does not decompose into a product of Vogan varieties, thus it is an example of the type of parameter this paper is dealing with.

Consider the group $\GL_4(F)$ and the Arthur parameter 
\[\psi(w,x,y)=\Sym^1(x) \oplus \Sym^1(y).\]
The infinitesimal parameter $\lambda_\psi$ is given by 
\[
\lambda_\psi(w) =
\begin{bmatrix} 
|w|^{1/2} & 0 & 0 & 0 \\
0 & |w|^{-1/2} & 0 & 0  \\
0 & 0 & |w|^{1/2} & 0 \\
0 & 0 & 0 & |w|^{1/2}
\end{bmatrix}.
\]
We can replace $\lambda_{\psi}$ by an element in its $\GL_4(\CC)$-conjugacy class - this will not change the geometry. Thus, we apply the permutation $(2\hspace{1mm}3)$ to $\lambda_{\phi_{\psi}}$  and drop the subscript $\psi$ from the notation to get
\[
\lambda(w) =
\begin{bmatrix} |w|^{1/2}&0&0&0\\0&|w|^{1/2}&0&0\\0&0&|w|^{-1/2}&0\\0&0&0&|w|^{-1/2}
\end{bmatrix}.
\]
Let us do the geometric side first. The above rearrangement enables us to easily compute the Vogan variety and the group action. 
\[V_{\lambda}=
\left\{ 
\begin{bmatrix}
0_{} & X \\
0_{} & 0_{}
\end{bmatrix}: X \in \op{Mat}_2(\CC) \right\}
\cong
 \op{Mat}_2(\CC).
\]
and \[H_{\lambda} = \GL_2(\CC) \times \GL_2(\CC).\]
The action of $H_\lambda$ on $V_{\lambda}$ is given by $$(g_1,g_2)\cdot X=g_1Xg_2^{-1}.$$ Thus, the rank of $X$ completely determines its $H_\lambda$-orbit. There are $3$ orbits - $C_0,C_1,$ and $C_2$ consisting of matrices of matrices of ranks $0,1$ and $2$ respectively. Note that $C_1$ is the orbit corresponding to $\phi_{\psi}$, so we set $C_{\psi}=C_1$.

In order to find the matrix $c_{\lambda}$ in this case, pick base points $x_0\in C_0$, $x_1\in C_1$ and $x_2\in C_2$ and consider the stalks of simple $\IC(\1_C)$. Since $\IC(\1_{C_0})$ is the skyscraper sheaf at $0\in V_\lambda$, its stalks are easy to compute: 
$\mathcal{H}^\bullet_{x_0}\IC(\1_{C_0}) = \1[0]$;
$\mathcal{H}^\bullet_{x_1}\IC(\1_{C_0}) = 0$; and
$\mathcal{H}^\bullet_{x_1}\IC(\1_{C_0}) = 0$.
Likewise, since $\IC(\1_{C_2})$ is the constant sheaf on $V_\lambda$ shifted by $4$, we have 
$\mathcal{H}^\bullet_{x}\IC(\1_{C_2}) = \1[4]$ for every $x\in V_\lambda$.
Only $\IC(\1_{C_1})$ is interesting - its stalks are given by
\begin{eqnarray*}
\mathcal{H}^\bullet_{x_0} \IC(\1_{C_1}) &=& H^\bullet(\mathbb{P}^1)[3] = \1[1] \oplus \1[3]\\
\mathcal{H}^\bullet_{x_1}\IC(\1_{C_1}) &=& \1[3] \\
\mathcal{H}^\bullet_{x_2}\IC(\1_{C_1}) &=& 0. 
\end{eqnarray*}
In $K\Per_{H_\lambda}(V_\lambda)$ we now have
\begin{eqnarray*}\,
[\IC(\1_{C_0})] &=& [\1_{C_0}^\natural] \\ \,
(-1)[\IC(\1_{C_1})] &=& \hskip-5pt 2 [\1_{C_0}^\natural] + [\1_{C_1}^\natural] \\ \,
[\IC(\1_{C_2})] &=& [\1_{C_0}^\natural] + [\1_{C_1}^\natural] + [\1_{C_2}^\natural].
\end{eqnarray*}
Therefore,
\[
c_\lambda=
\begin{bmatrix}
1 & 2 & 1\\
0 & 1 & 1\\
0 & 0 & 1
\end{bmatrix}.
\]
On the spectral side, $\Rep^\text{fl}_\lambda(G)$ contains exactly three irreducible representations, all appearing in $\Ind_{B}^{G}(\chi_\lambda)$ for $\chi_\lambda(\diag(t_1, t_2,t_3,t_4)) = |t_1 t_2|^{1/2}|t_3 t_4|^{-1/2}$. The these three irreducible representations are: the unique irreducible quotient $\pi_0$ of $\Ind_{B}^{G}(\chi)$ and two irreducible subrepresentations $\pi_1$ and $\pi_2$, of which only $\pi_2$ is tempered. Note that $\pi_{\psi}=\pi_1$ as it corresponds to the parameter $\phi_{\psi}$. The standard representations for $\pi_0$, $\pi_1$ and $\pi_2$ are given as follows, where $P_0$ is the standard Borel, $P_1$ is the standard parabolic with Levi $\GL_2\times \GL_1^2$ and $P_2$ is the standard parabolic with Levi $\GL_2\times \GL_2$:
\begin{eqnarray*}
\Delta(\pi_0) &=& \Ind_{P_0}^{G}(\chi_\lambda) \\
\Delta(\pi_1) &=& \Ind_{P_1}^{G}(\St_2\otimes\chi_1) \\
\Delta(\pi_0) &=& \Ind_{P_2}^{G}(\St_2\otimes \St_2) .
\end{eqnarray*}
Here, $\St_2$ is Steinberg for $\GL_2(F)$ and $\chi_1$ is the character $\chi$ appearing in Example~\ref{example:Steinberg GL2}.
%
In fact, $\Ind_{B}^{G}(\chi)$ is a length-four representation and $\pi_1$ appears with multiplicity $2$ in $\Ind_{B}^{G}(\chi)$:
\[
    [\Delta(\pi_0)] =  [\pi_0] + 2[\pi_1] + [\pi_2]
\]
Being tempered, $\pi_2$ is its own standard representation, so $[\Delta(\pi_2)] = [\pi_2]$.
In fact, 
\[
\pi_2 = \Ind_{P_2}^{G}(\St_{\GL_2}\otimes\St_{\GL_2}),
\]
which can be used to see that
\[
\begin{tikzcd}
    0 \arrow{r} & \pi_2 \arrow{r} & \Delta(\pi_1) \arrow{r} & \pi_1 \arrow{r} & 0,
\end{tikzcd}
\]
is a short exact sequence in $\Rep^\text{fl}_\lambda(G)$. We therefore have
\begin{eqnarray*}\,
    [\Delta(\pi_1)] &=& [\pi_1] + [\pi_2]
\end{eqnarray*}
in $K\Rep^\text{fl}_{\lambda}(G)$.
Now we know the matrix $m$:
\[
m_\lambda = 
\begin{bmatrix}
1 & 0 & 0\\
2 & 1 & 0\\
1 & 1 & 1
\end{bmatrix}.
\]
We see that $m_{\lambda} = \,^tc_{\lambda}$, yet another demonstration of Hypothesis~\ref{hypothesis}.

\end{example}

\begin{example}\label{example: levi braden gl4}
In this example, we discuss the geometry of the Arthur parameter for the Levi subgroup carved out by the parameter in Example \ref{example:braden gl4}, and compute $c_{\lambda_M}$ and $m_{\lambda_M}$.
Once again, we consider the Arthur parameter from Example \ref{example:braden gl4}:
\[\psi(w,x,y)=\Sym^1(x)\oplus \Sym^1(y)\]
Following the recipe of Section~\ref{introduction}, this picks out the Levi subgroup $M=\GL_2\times \GL_2$ of $\GL_4$, with associated simple Arthur parameters $\psi_1(w,x,y)= \Sym^1(x)\otimes \Sym^0(y)$ and $\psi_2(w,x,y)= \Sym^0(x)\otimes \Sym^1(y)$ which correspond in turn to Langlands parameters $\phi_1, \phi_2$ and infinitesimal parameters $\lambda_1,\lambda_2$, respectively. We use $\psi_M$ and $\lambda_M$ to denote the Arthur parameter for $M$. 

Then $V_{\lambda_M}=V_{\lambda_1} \times_{} V_{\lambda_2}$ consists of elements of the type 
\[(x_0,x_1) \coloneqq \left(\begin{bmatrix}0&x_0\\0&0\end{bmatrix} , \begin{bmatrix}0&x_1\\0&0\end{bmatrix}\right),\]
the group $H_{\lambda_M}=H_{\lambda_1} \times_{} H_{\lambda_1}$ consists of elements of the type 
\[
(t,s) \coloneqq \left(\begin{bmatrix}t_1&0\\0&t_2\end{bmatrix} , \begin{bmatrix}s_1&0\\0&s_2\end{bmatrix}\right),\]
and the action is given by
\[
(t,s)\cdot (x_0,x_1) = (t_1x_0t_2^{-1}, s_1x_1s_2^{-1}).
\]
There are four orbits depending on whether or not $x_i=0$ for $i=0,1$. We denote them as $C_{00}, C_{10}, C_{01}$ and $C_{11}$, where $C_{ij}$ corresponds to the $H_{\lambda_M}$-orbit of $(x_0,x_1)=(i,j)$.

To identify $V_{\lambda_M}$ as a subspace of $V_{\lambda}$,
we consider each element in $V_{\lambda_M}$ as a block diagonal $4 \times 4$ matrix in the obvious way, then apply the same permutation $(2\hspace{1mm}3)$ as before. The variety is still denoted $V_{\lambda_M}$ and its elements are identified with matrices of the type
\[\begin{bmatrix}0_{2 \times 2}&X\\0_{2\times 2}&0_{2\times 2}\end{bmatrix} \text{ where } X=\begin{bmatrix}x_0&0\\0&x_1\end{bmatrix}.\]
We do the same to elements of $H_{\lambda_M}$ to identify it with a torus in $\GL_4$,  with elements as matrices
\[\begin{bmatrix}\operatorname{diag}(t_1,s_1)&0_{2 \times 2}\\0_{2\times 2}&\operatorname{diag}(t_2,s_2)\end{bmatrix} \text{ where } t_i,s_i \in \CC^{\times}.\]
The conjugation action still takes $x_0 \mapsto t_1x_0t_2^{-1}$; likewise $x_1 \mapsto s_1x_1s_2^{-1}$. Thus, the embedding $V_{\lambda_M} \xhookrightarrow{} V_{\lambda}$ is $H_{\lambda_M}$-equivariant.
We continue to use the notation $C_{ij}$ for the $H_{\lambda_M}$-orbits . The orbit $C_{\psi_M}$ of type $\psi_M$ is $C_{10}$. At this point, we encourage the reader to think about the restriction of the orbits $C_0,C_1,C_2$ from Example \ref{example:braden gl4} to $V_{\lambda_M}$. In particular, observe that $C_1$ restricts to $C_{10} \sqcup C_{01}$. 

Let us compute the matrix $c_{\lambda_M}$. From Example \ref{example:Steinberg GL2} we know \[c_{\lambda_i}= 
\begin{bmatrix}
    1&1\\0&1
\end{bmatrix},
\]
for $i=1,2$. This matrix can be interpreted as a change of basis matrix from the standard sheaves to shifted simple perverse sheaves. One easily sees that
\[c_{\lambda_M}=c_{\lambda_1}\otimes c_{\lambda_2}
= 
\begin{bmatrix}
    1&1\\0&1
\end{bmatrix}\otimes \begin{bmatrix}
    1&1\\0&1
\end{bmatrix}
=\begin{bmatrix}
    1&1&1&1\\
    0&1&0&1\\
    0&0&1&1\\
    0&0&0&1
\end{bmatrix}.
\]
This is a change of basis matrix from the shifted simple perverse sheaves \[
\{(-1)^{\dim C_{ij}}[\mathcal{IC}(\1_{C_{ij}})]: 0\leq i,j \leq 1\}
\]
to standard sheaves 
\[
\{[\1_{C_{ij}}^{\natural}]: 0\leq i,j \leq 1\}
\]
in $K\Per_{H_{\lambda_M}}(V_{\lambda_M})$. 

The change of basis matrix from standard sheaves to shifted simple perverse sheaves is therefore 
\[
c^{-1}_{\lambda_M}=\begin{bmatrix}
    1&-1&-1&-1\\0&1&0&-1\\0&0&1&-1\\0&0&0&1
\end{bmatrix}.
\]
On the spectral side, the orbits $C_{ij}$ correspond to irreducible representations $\pi_{ij}$ in $\Rep^{\op{fl}}_{\lambda_M}(M)$. Following an analogous calculation for the spectral multiplicity matrix using $m$ from Example \ref{example:Steinberg GL2}, we get 
\[m_{\lambda_M}^{-1}=\begin{bmatrix}
1 & 0 & 0 & 0 \\
0 & 1 & 1 & 0 \\
0 & 0 & 0 & 1 
\end{bmatrix}.\]
\end{example}
\begin{example}\label{Lift and pullback for gl4 braden and levi}
In this example we compute the endoscopic lifts of some irreducible representations, using only the definition of $\Lift_M^G$ from Section~\ref{ssec:Lift}, with $G=GL_4$, and $\psi$ and $M$ as in Examples \ref{example:braden gl4} and \ref{example: levi braden gl4}.

First, we calculate restriction of standard sheaves via
\[\varepsilon^*: K_{\CC}\Per_{H_{\lambda}}(V_{\lambda}) \to K_{\CC}\Per_{H_{\lambda_M}}(V_{\lambda_M}). \]
We note that 
\begin{eqnarray*}\, 
[\1_{C_0}^\natural \lvert_{V_{\lambda_M}}] &=& [\1_{C_{00}}^\natural]\\ \,
[\1_{C_1}^\natural\lvert_{V_{\lambda_M}}] &=& [\1^\natural_{C_{10}}]+[\1^\natural_{C_{01}}]\\ \,
[\1_{C_2}^\natural\lvert_{V_{\lambda_M}}] &=& [\1^\natural_{C_{11}}],
\end{eqnarray*}
because, respectively, $C_0\cap V_{\lambda_M} =C_{00}$, $C_1 \cap V_{\lambda_M} = C_{10} \sqcup C_{01}$, $C_2\cap V_{\lambda_M} = C_{11}$. 
Thus, the matrix for $\varepsilon^*$ with respect to the basis of standard sheaves is 
\[
[\varepsilon^*]_{\op{sts}}
=
\begin{bmatrix}
    1 & 0 & 0\\
    0 & 1 & 0\\
    0 & 1 & 0\\
    0 & 0 & 1
\end{bmatrix}.
\] 
From Example \ref{example:braden gl4}, recall that $c_{\lambda}$ is a change of basis matrix from shifted simple perverse sheaves to standard sheaves in $K\Per_{H_{\lambda}}(V_{\lambda})$. 
From Example~\ref{example: levi braden gl4} recall that $c^{-1}_{\lambda_M}$ is the change of basis matrix from standard sheaves to shifted simple perverse sheaves in $K\Per_{H_{\lambda_M}}(V_{\lambda_M})$.
Therefore, restriction of shifted simple objects from $K\Per_{H_{\lambda}}(V_{\lambda})$ to $K\Per_{H_{\lambda_M}}(V_{\lambda_M})$ via $\varepsilon^*$ is given by
\begin{align*}
[\varepsilon^*]_{\op{ssim}}
&=
c^{-1}_{\lambda_M} \ [\varepsilon^*]_{\op{sts}}\ c_{\lambda} \\ 
&= 
\begin{bmatrix}
    1 & -1 & -1 & 1\\
    0 & 1 & 0 & -1\\
    0 & 0 & 1 & -1\\
    0 & 0 & 0 & 1
\end{bmatrix}
\ 
\begin{bmatrix}
    1 & 0 & 0\\
    0 & 1 & 0\\
    0 & 1 & 0\\
    0 & 0 & 1
\end{bmatrix}
\
\begin{bmatrix}
1 & 2 & 1\\
0 & 1 & 1\\
0 & 0 & 1
\end{bmatrix} 
=
\begin{bmatrix}
    1 & 0 & 0\\
    0 & 1 & 0\\
    0 & 1 & 0\\
    0 & 0 & 1
\end{bmatrix}.
\end{align*}
This shows:
\begin{align*}
    \varepsilon^*[\IC(\1_{C_0})] &= [\IC(\1_{C_{00}})], \\
    \varepsilon^*[\IC(\1_{C_1})[-3]] &=  [\IC(\1_{C_{01}})[-1]] + [\IC(\1_{C_{01}})[-1]], \\
    \varepsilon^*[\IC(\1_{C_2})[-4]] &= [\IC(\1_{C_{11}})[-2]],
\end{align*}
so
\begin{align*}
    \varepsilon^*[\IC(\1_{C_0})] &= [\IC(\1_{C_{00}})], \\
    \varepsilon^*[\IC(\1_{C_1})] &=  [\IC(\1_{C_{01}})] + [\IC(\1_{C_{01}})], \\
    \varepsilon^*[\IC(\1_{C_2}) &= [\IC(\1_{C_{11}})].
\end{align*}

Now let us use this calculation, together with Hypothesis~\ref{hypothesis}, to calculate $\Lift_M^G[\sigma]$ for every $\sigma \in \Pi_{\lambda_M}(M)$:
\begin{align*}
[\Lift_M^G]_{\op{sim}}
&= \,^t[\varepsilon^*]_{\op{ssim}} \\
&= \,^t\left(
c^{-1}_{\lambda_M} \ [\varepsilon^*]_{\op{sts}}\ c_{\lambda}\right) \\
&= \,^t c_{\lambda}\ \,^t [\varepsilon^*]_{\op{sts}}\ \,^tc_{\lambda_M}^{-1} \\
&= m_{\lambda}\ \,^t [\varepsilon^*]_{\op{sts}}\ m_{\lambda_M}^{-1}\\
&= \begin{bmatrix}
1 & 0 & 0\\
2 & 1 & 0\\
1 & 1 & 1
\end{bmatrix} 
\begin{bmatrix}
1 & 0 & 0 & 0 \\
0 & 1 & 1 & 0 \\
0 & 0 & 0 & 1 
\end{bmatrix}
\begin{bmatrix}
    1 & 0 & 0 & 0\\
    -1 & 1 & 0 & 0\\
    -1 & 0 & 1 & 0\\
    1 & -1 & -1 & 1
\end{bmatrix}
= 
\begin{bmatrix}
1 & 0 & 0 & 0 \\
0 & 1 & 1 & 0 \\
0 & 0 & 0 & 1 
\end{bmatrix}.
\end{align*}
This shows:
\begin{align*}
    \Lift_M^G[\pi_{00}] &= [\pi_0] \\
    \Lift_M^G[\pi_{01}] &= [\pi_1] \\
    \Lift_M^G[\pi_{10}] &= [\pi_1] \\
    \Lift_M^G[\pi_{11}] &= [\pi_2] .
\end{align*}
\end{example}
In particular, $\Lift_M^G[\pi_{10}]=[\pi_1]$, in other words,     $\Lift_M^G[\pi_{\psi_M}] = [\pi_{\psi}]$, which is true in general for $\GL_n$ and we prove this in Proposition \ref{prop: Lift on irreducibles of A type}.


\begin{bibdiv}
\begin{biblist}

\bib{Achar:book}{book}{
   author={Achar, Pramod N.},
   title={Perverse sheaves and applications to representation theory},
   series={Mathematical Surveys and Monographs},
   volume={258},
   publisher={American Mathematical Society, Providence, RI},
   date={2021},
   pages={xii+562},
   isbn={978-1-4704-5597-2},
}

\bib{ABV}{book}{
   author={Adams, Jeffrey},
   author={Barbasch, Dan},
   author={Vogan, David A., Jr.},
   title={The Langlands classification and irreducible characters for real reductive groups},
   series={Progress in Mathematics},
   volume={104},
   publisher={Birkh\"{a}user Boston, Inc., Boston, MA},
   date={1992},
   pages={xii+318},
}

\bib{Arthur:book}{book}{
   author={Arthur, James},
   title={The endoscopic classification of representations},
   series={American Mathematical Society Colloquium Publications},
   volume={61},
   note={Orthogonal and symplectic groups},
   publisher={American Mathematical Society, Providence, RI},
   date={2013},
   pages={xviii+590},
   isbn={978-0-8218-4990-3},
   doi={10.1090/coll/061},
}

\bib{Arthur:unipotent-motivation}{article}{
   author={Arthur, James},
   title={Unipotent automorphic representations: global motivation},
   conference={
      title={Automorphic forms, Shimura varieties, and $L$-functions, Vol.
      I},
      address={Ann Arbor, MI},
      date={1988},
   },
   book={
      series={Perspect. Math.},
      volume={10},
      publisher={Academic Press, Boston, MA},
   },
   date={1990},
   pages={1--75},
   review={\MR{1044818}},
}

\bib{arthur1989unipotent}{article}{
  title={Unipotent automorphic representations: conjectures},
  author={Arthur, James},
  year={1989},
   book={
      series={Ast\'{e}risque},
      volume={},
      publisher={},
   },
}


\bib{Atobe}{article}{
   author={Atobe, Hiraku},
   title={Construction of local $A$-packets},
   journal={J. Reine Angew. Math.},
   volume={790},
   date={2022},
   pages={1--51},
   issn={0075-4102},
   review={\MR{4472864}},
   doi={10.1515/crelle-2022-0030},
}

\bib{BBD}{article}{
   author={Be\u{\i}linson, A. A.},
   author={Bernstein, J.},
   author={Deligne, P.},
   title={Faisceaux pervers},
   conference={
      title={Analysis and topology on singular spaces, I},
      address={Luminy},
      date={1981},
   },
   book={
      series={Ast\'{e}risque},
      volume={100},
      publisher={Soc. Math. France, Paris},
   },
   date={1982},
   pages={5--171},
}

\bib{Borel:Corvallis}{book}{
   author={Borel, A},
   title={Automorphic L-functions},
   series={Automprhic forms, representations, and $L$-functions Part 2 (Proc. Sympos. Pure Math, Corvallis XXXIII)},
   note={p. 27-61},
   publisher={American Mathematical Society, Providence, RI},
   date={1979},
   pages={27-61}
}


\bib{CG}{book}{
   author={Chriss, Neil},
   author={Ginzburg, Victor},
   title={Representation theory and complex geometry},
   series={Modern Birkh\"{a}user Classics},
   note={Reprint of the 1997 edition},
   publisher={Birkh\"{a}user Boston, Ltd., Boston, MA},
   date={2010},
   pages={x+495},
   isbn={978-0-8176-4937-1},
   doi={10.1007/978-0-8176-4938-8},
}

\bib{Clozel}{article}{
   author={Clozel, Laurent},
   title={Sur une conjecture de Howe. I},
   language={English, with French summary},
   journal={Compositio Math.},
   volume={56},
   date={1985},
   number={1},
   pages={87--110},
   issn={0010-437X},
}

\bib{CFMMX}{book}{
   author={Cunningham, Clifton},
   author={Fiori, Andrew},
   author={Moussaoui, Ahmed},
   author={Mracek, James},
   author={Xu, Bin},
   title={A-packets for p-adic groups by way of microlocal vanishing cycles of perverse sheaves, with examples},
   series={Memoirs of the American Mathematical Society},
   volume={276},
   date={2022},
   number={1353},
}



\bib{CFZ:cubics}{article}{
author={Cunningham, Clifton},
author={Fiori, Andrew},
author={Zhang, Qing},
title={A-packets for $G_2$ and perverse sheaves on cubics},
journal={Advances in mathematics},
volume={395},
date={2022},
}


\bib{CFZ:unipotent}{unpublished}{
author={Cunningham, Clifton},
author={Fiori, Andrew},
author={Zhang, Qing},
title={Toward the endoscopic classification of unipotent representations of $p$-adic $G_2$},
note={http://arxiv.org/abs/2101.04578},
date={2021},
}

\bib{CFK}{article}{
title={Appearance of the Kashiwara-Saito singularity in the representation theory of $p$-adic GL(16)},
author={Cunningham, Clifton},
author={Fiori, Andrew},
author={Kitt, Nicole},
journal={Pacific Journal of Mathematics},
date={2023},
note={https://arxiv.org/abs/2103.04538},
}


\bib{CR:irred}{unpublished}{
   author={Cunningham, Clifton},
   author={Ray, Mishty},
   title={Proof of Vogan's conjecture on A-packets: irreducible parameters for $p$-adic general linear groups},
   date={2022},
   note={https://arxiv.org/abs/2206.01027}
}



\bib{vanDijk}{article}{
   author={van Dijk, G.},
   title={Computation of certain induced characters of ${\germ p}$-adic
   groups},
   journal={Math. Ann.},
   volume={199},
   date={1972},
   pages={229--240},
   issn={0025-5831},
   doi={10.1007/BF01429876},
}



\bib{GM:book}{book}{
   author={Goresky, Mark},
   author={MacPherson, Robert},
   title={Stratified Morse theory},
   series={Ergebnisse der Mathematik und ihrer Grenzgebiete (3) [Results in
   Mathematics and Related Areas (3)]},
   volume={14},
   publisher={Springer-Verlag, Berlin},
   date={1988},
   pages={xiv+272},
   isbn={3-540-17300-5},
   doi={10.1007/978-3-642-71714-7},
}

\bib{GM:Lefschetz}{article}{
   author={Goresky, Mark},
   author={MacPherson, Robert},
   title={Local contribution to the Lefschetz fixed point formula},
   journal={Invent. Math.},
   volume={111},
   date={1993},
   number={1},
   pages={1--33},
   issn={0020-9910},
   doi={10.1007/BF01231277},
}

\bib{Illusie}{article}{
   author={Illusie, Luc},
   title={Around the Thom-Sebastiani theorem, with an appendix by Weizhe
   Zheng},
   journal={Manuscripta Math.},
   volume={152},
   date={2017},
   number={1-2},
   pages={61--125},
}






\bib{KZ}{article}{
   author={Knight, Harold},
   author={Zelevinsky, Andrei},
   title={Representations of quivers of type $A$ and the multisegment
   duality},
   journal={Adv. Math.},
   volume={117},
   date={1996},
   number={2},
   pages={273--293},
   issn={0001-8708},
}

\bib{knight1996representations}{article}{
  title={Representations of quivers of type A and the multisegment duality},
  author={Knight, Harold},
  author={Zelevinsky, Andrei},
  journal={Advances in mathematics},
  volume={117},
  number={2},
  pages={273--293},
  year={1996},
  publisher={New York: Academic Press, 1965-}
}

\bib{Konno}{article}{
   author={Konno, Takuya},
   title={A note on the Langlands classification and irreducibility of
   induced representations of $p$-adic groups},
   journal={Kyushu J. Math.},
   volume={57},
   date={2003},
   number={2},
   pages={383--409},
   issn={1340-6116},
   doi={10.2206/kyushujm.57.383},
}

\bib{kudla1994local}{article}{
  title={The local Langlands correspondence: the non-archimedean case},
  author={Kudla, Stephen S},
  journal={Motives (Seattle, WA, 1991)},
  volume={55},
  number={Part 2},
  pages={365--391},
  year={1994},
  publisher={American Mathematical Society Providence, RI}
}
\bib{langlands-shelstad}{article}{
   author={Langlands, R. P.},
   author={Shelstad, D.},
   title={On the definition of transfer factors},
   journal={Math. Ann.},
   volume={278},
   date={1987},
   number={1-4},
   pages={219--271},
   issn={0025-5831},
   review={\MR{909227}},
   doi={10.1007/BF01458070},
}

\bib{Lusztig:Cuspidal2}{article}{
   author={Lusztig, George},
   title={Cuspidal local systems and graded Hecke algebras. II},
   note={With errata for Part I [Inst. Hautes \'{E}tudes Sci. Publ. Math. No. 67
   (1988), 145--202;  MR0972345 (90e:22029)]},
   conference={
      title={Representations of groups},
      address={Banff, AB},
      date={1994},
   },
   book={
      series={CMS Conf. Proc.},
      volume={16},
      publisher={Amer. Math. Soc., Providence, RI},
   },
   date={1995},
   pages={217--275},
}

\bib{Massey}{article}{
   author={Massey, David B.},
   title={The Sebastiani-Thom isomorphism in the derived category},
   journal={Compositio Math.},
   volume={125},
   date={2001},
   number={3},
   pages={353--362},
}

\bib{MW:involution}{article}{
   author={M\oe glin, Colette},
   author={Waldspurger, Jean-Loup},
   title={Sur l'involution de Zelevinski},
   journal={J. Reine Angew. Math.},
   volume={372},
   date={1986},
   pages={136--177},
   issn={0075-4102},
}

\bib{Mracek}{book}{
   author={Mracek, James},
   title={Applications of Algebraic Microlocal Analysis in Symplectic
   Geometry and Representation Theory},
   note={Thesis (Ph.D.)--University of Toronto (Canada)},
   publisher={ProQuest LLC, Ann Arbor, MI},
   date={2017},
   pages={101},
   isbn={978-0355-53067-4},
}

\bib{Pyasetskii}{article}{
	Author = {Pjasecki\u\i , V. S.},
	journal = {Akademija Nauk SSSR. Funkcional\cprime nyi Analiz i ego Prilo\v zenija},
	Issn = {0374-1990},
	Number = {4},
	Pages = {85--86},
	Title = {Linear {L}ie groups that act with a finite number of orbits},
	Volume = {9},
	Year = {1975}
}


\bib{Riddlesden}{article}{
author={Riddlesden, Connor},
title={Combinatorial approach to ABV-packets for $GL_n$},
journal={MSc thesis, University of Lethbridge},
year={2022},
note={\url{https://opus.uleth.ca/handle/10133/6377}},
}


\bib{Solleveld:pKLH}{unpublished}{
author={Solleveld, Maarten},
title={Graded Hecke algebras, constructible sheaves and the p-adic Kazhdan--Lusztig conjecture},
note={\href{https://arxiv.org/pdf/2106.03196.pdf}{https://arxiv.org/pdf/2106.03196.pdf}},
date={2022}
}


\bib{Vogan:Langlands}{article}{
   author={Vogan, David A., Jr.},
   title={The local Langlands conjecture},
   conference={
      title={Representation theory of groups and algebras},
   },
   book={
      series={Contemp. Math.},
      volume={145},
      publisher={Amer. Math. Soc., Providence, RI},
   },
   date={1993},
   pages={305--379},
}

\bib{Vooys}{unpublished}{
author={Vooys, Geoff},
title={Equivariant Functors and Sheaves},
note={\href{https://arxiv.org/pdf/2110.01130.pdf}{https://arxiv.org/pdf/2110.01130.pdf}},
date={2023}
}




\bib{Z2}{article}{
   author={Zelevinsky, Andrei V.},
   title={Induced representations of reductive ${\germ p}$-adic groups. II.
   On irreducible representations of ${\rm GL}(n)$},
   journal={Ann. Sci. \'{E}cole Norm. Sup. (4)},
   volume={13},
   date={1980},
   number={2},
   pages={165--210},
}

\bib{zelevinskii1981p}{article}{
  title={p-adic analog of the kazhdan-Lusztig hypothesis},
  author={Zelevinskii, Andrei Vladlenovich},
  journal={Functional Analysis and Its Applications},
  volume={15},
  number={2},
  pages={83--92},
  year={1981},
  publisher={Springer}
}

\end{biblist}
\end{bibdiv}
\end{document}